\documentclass[letterpaper,12pt,reqno]{amsart}
\usepackage[left=1in,right=1in,top=0.9in,bottom=0.9in]{geometry}
\usepackage[utf8]{inputenc}
\usepackage[dvipsnames]{xcolor}
\usepackage{amsmath,amssymb,amsfonts,amsthm,hyperref}
\usepackage{cleveref}
\usepackage{cite}
\usepackage{tikz-cd}
\usepackage{enumitem}
\theoremstyle{definition}
\newtheorem{definition}{Definition}
\newtheorem{theorem}[definition]{Theorem}
\newtheorem{lemma}[definition]{Lemma}
\newtheorem{remark}[definition]{Remark}
\newtheorem{example}[definition]{Example}
\newtheorem{conjecture}[definition]{Conjecture}
\newtheorem{proposition}[definition]{Proposition}
\usepackage{tikz-3dplot}
\usetikzlibrary{calc,intersections,through,backgrounds}
\definecolor{q1}{RGB}{160, 0, 0}
\definecolor{q2}{RGB}{0,0,160}
\definecolor{q3}{RGB}{0, 160, 0}
\definecolor{c1}{RGB}{0, 160, 160}
\definecolor{c2}{RGB}{160, 0, 160}
\definecolor{middlecolor}{RGB}{82, 62, 62}
\usetikzlibrary{positioning,calc}
\usetikzlibrary{intersections}

\newcommand{\qbarcheck}[1]{\bar{Q}_{#1}}

\newcommand{\ZZ}{\mathbf Z}
\newcommand{\CC}{\mathbf C}

\newcommand{\PP}{\mathbf P}
\newcommand{\FF}{\mathbf F}
\newcommand{\Rho}{\mathrm{P}}
\newcommand{\pp}{\mathbf{p}}

\newcommand{\bA}{\mathbf{A}}
\newcommand{\bB}{\mathbf{B}}
\newcommand{\bQ}{\mathbf{Q}}

\newcommand{\qq}{\mathbf{q}}
\newcommand{\ba}{\mathbf{a}}
\newcommand{\calG}{\mathcal G}

\newcommand{\row}[1]{\mathbf{r}_{#1}}
\newcommand{\si}{\sigma}
\newcommand{\ratquo}[3]{#1 / #2 \cong_{\textbf{Bir}} #3}

\DeclareMathOperator{\rowspan}{rowspan}
\DeclareMathOperator{\rank}{rank}
\DeclareMathOperator{\diag}{diag}
\DeclareMathOperator{\Gr}{Gr}

\DeclareMathOperator{\GL}{GL}
\DeclareMathOperator{\ED}{ED}

\newcommand{\GammaAqp}{\Gamma_{\bA, \qq , \pp}}

\newcommand{\GammaAbarqp}{\Gamma_{\bar{\bA}, \qq , \pp}}
\newcommand{\GammaAbarp}{\Gamma_{\bar{\bA} , \pp}}

\newcommand{\Gammaqbarp}{\Gamma_{\bar{\qq} , \pp}}
\newcommand{\Rhoaqp}{\Rho_{\ba, \qq, \pp}}

\newcommand{\qqbar}{\bar{\qq}}
\newcommand{\AAbar}{\bar{\bA}}

\newcommand{\Qbar}{\bar{\bQ}}
\newcommand{\Ifoc}[1]{{I}_{#1}}
\newcommand{\foc}{\Ifoc{}}
\newcommand{\sixB}{\Ifoc{6..12} (\bB)}
\newcommand{\sixBstar}{\Ifoc{6..12} ( \bB^\star)}
\newcommand{\mB}{\Ifoc{m}(\bB )}

\newcommand{\mBstar}{\Ifoc{m} (\bB^\star )}%

\crefname{assumption}{Assumption}{assumption}

\DeclareMathOperator{\vectorize}{vec}
\DeclareMathOperator{\PGL}{PGL}

\title{Algebra and Geometry of Camera Resectioning}
\author{Erin Connelly, Timothy Duff, Jessie Loucks-Tavitas}
\newcommand{\Addresses}{{
  \footnotesize
 \textsc{Department of Mathematics, University of Washington}\par\nopagebreak
   \medskip
  \textit{Email address}: E.~Connelly \texttt{erin96@uw.edu}\par
  \medskip
  \textit{Email address}: T.~Duff \texttt{timduff@uw.edu}\par
  \medskip
  \textit{Email address}: J.~Loucks-Tavitas \texttt{jaloucks@uw.edu}
  \medskip
}}

\begin{document}
\maketitle
\tableofcontents 
\begin{abstract}
We study algebraic varieties associated with the camera resectioning problem.
We characterize these resectioning varieties' multigraded vanishing ideals using Gr\"{o}bner basis techniques. 
As an application, we derive and re-interpret celebrated results in geometric computer vision related to camera-point duality.
We also clarify some relationships between the classical problems of optimal resectioning and triangulation, state
a conjectural formula for the Euclidean distance degree of the resectioning variety, and discuss how this conjecture relates to the recently-resolved multiview conjecture.
\end{abstract}

\section{Introduction}\label{sec:intro}

The \emph{dramatis personae} of the classical pinhole camera model are a full-rank $3\times 4$ matrix $A$ representing a camera, a $4\times 1$ matrix $q$ representing a world point, and a $3\times 1$ matrix $p$ representing its projection into an image.
Image formation may be understood via the projective-linear map
\begin{equation}\label{eq:cam}
\begin{split}
A : \PP^3 &\dashrightarrow \PP^2 \\
q &\mapsto A q,
\end{split}
\end{equation}
and we write $A q \sim p$ if these two vectors represent the same point in $\PP^2.$
The \emph{center} of the camera $A$ is the unique point where the map~\eqref{eq:cam} is undefined.

The pinhole camera, despite its simplicity, remains a good model of physical cameras. 
This explains its importance in modern computer vision applications such as structure-from-motion (SfM) and Simultaneous Localization and Mapping (SLAM).
On the other hand, classical problems associated with 3D reconstruction have been studied long before the advent of computers, and the role played by algebraic methods in their solution has long been apparent.
For instance, Hesse in 1863 formulated the problem of constructing two homographic configurations of 7 lines in space, each prescribed to pass through a configuration of 7 points in a plane~\cite{7pt-Hesse}.
Hesse's reduction of this problem to computing the roots of a cubic equation may be understood as an early instance of the so-called 7 point algorithm. 
Similarly, Grunert's 1841 ``3D Pothenot problem"~\cite{Grunert-1841} is known nowadays as the perspective 3-point (P3P) problem, and his general strategy reducing the problem to a quartic equation remains in use today.

In recent years, the name \emph{algebraic vision}~\cite{https://doi.org/10.48550/arxiv.2210.11443} has been coined to describe a body of interdisciplinary research in which notions from algebra and vision flow freely.
To date, algebraic vision has largely focused on problems which we refer to as the \emph{full reconstruction problem} and \emph{triangulation}.

In the full reconstruction problem, we are given a collection of image points $\tilde{p}_{1 1}, \ldots , \tilde{p}_{m n} $, and our task is to recover a set of cameras $A_1, \ldots , A_m$ and world points $q_1, \ldots , q_n$ that is consistent with these observations.
Hesse's solution treats the ``minimal" case $(m,n) = (2,7)$.
Today, there are many works which solve analogous minimal problems which can be used effectively in SfM pipelines (see eg.~\cite{DBLP:conf/cvpr/LarssonOAWKP18,PLMP,PL1P,kukelova2008automatic,larsson2017efficient}.)

In triangulation, we are given not only image points, but also the cameras that produced them, $\bar{A}_1, \ldots , \bar{A}_m$. We need only recover one or more unknown world points.
Already for $m=2,$ an exact solution to this problem will typically not exist, due to the fact that the lines in $\PP^3$ projecting to generic image points under $\bar{A}_1, \bar{A}_2$ will be skew.
Nevertheless, algebraic methods have led to a wealth of knowledge about the triangulation problem.
For example, the \emph{multiview ideal} associated to $\bar{A}_1, \ldots , \bar{A}_m$ gives rise to a complete set of algebraic constraints on \emph{any} $m$-tuple of image points they produce.
There is a considerable literature related to multiview ideals ~\cite{idealMultiview,Hilb,DBLP:conf/iccv/FaugerasM95,HA97}.
\Cref{thm:multiview-omnibus} collects some important previous results. 

Often regarded as being ``dual" to triangulation is the problem of camera \emph{resectioning}.
Here, we assume $n$ image points are given along with the configuration of world points $\bar{\qq} = (\bar{q}_1, \ldots , \bar{q}_n) \in (\PP^2)^n$ from which they were produced by a single unknown camera $A.$
Grunert's 1841 paper gives a minimal solution for $(m,n) = (1,3)$ under the assumption that $A$ is \emph{Euclidean}.
Without this assumption, $A$ is a general $3\times 4$ matrix, and we need $n\ge 6$.

\subsection{Results and Organization}\label{subsec:results-organization}

In this paper, we aim to bring the general resectioning problem up-to-speed with the latest developments in algebraic vision.
In~\Cref{sec:res-triang}, after recalling some previous results about multiview varieties, we state our first main result,~\Cref{thm:main-ideals}. 
This characterizes a complete set of algebraic constraints for the resectioning problem, under the genericity assumption that no four of the given world points are coplanar.
These constraints are given by $k$-linear polynomials for $6\le k \le 12$ which generate  the \emph{resectioning ideal} $I( \Gammaqbarp^{m,n})$ (\Cref{defn:resectioning}).
Our work is a natural continuation of recent work by Agarwal et al.~\cite{agarwal2022atlas}, and we resolve three of its open questions.
For instance,~\Cref{thm:main-ideals} resolves~\cite[\S 8.1, Q4]{agarwal2022atlas} for generic $\bar{\qq}$ by determining a universal Gr\"{o}bner basis for $I( \Gammaqbarp^{m,n}).$ 

We note that resectioning ideals have several pleasant properties from the point of view of commutative algebra: namely, for generic $\bar{\qq} \in \left( \PP^3 \right)^n$,
\begin{enumerate}
\item For fixed $m$ and $n,$ resectioning ideals are homogeneous with respect to a natural $\ZZ^{mn}$-grading, and have the same $\ZZ^{mn}$-graded Hilbert function as long as no four points are coplanar.
\Cref{prop:set-theoretic} implies that this Hilbert function may be obtained by specializing a combinatorial formula of Li~\cite[Theorem 1.1]{Li-IMRN}, based on the inclusion-exclusion rule.
Our ideal-theoretic result also considerably strengthens Li's set-theoretic description, and reduces the degrees of the equations that are needed.
\item 
The multidegrees of resectioning ideals are always equal to $1.$
A geometric explanation of this phenomenon follows along the lines explained in~\cite[\S 4]{breiding2022line}. 
See also~\cite[Theorem 4.2]{CCF23} for an explanation using multigraded Rees algebras.
\item For any monomial order $<$, the initial ideal $\operatorname{in}_< (I(\Gamma_{\bar{\qq}, \pp}^{m,n}))$ and the multigraded generic initial ideal $\operatorname{gin}_< (I(\Gamma_{\bar{\qq}, \pp}^{m,n}))$, although not equal as in the case of multiview ideals~\cite{Hilb}, are both radical.
In particular, $I(\Gamma_{\bar{\qq}, \pp}^{m,n})$ belongs to the class of \emph{Cartwright-Sturmfels ideals}, recently surveyed by Conca, De Negri, and Gorla~\cite{conca-survey}.
\end{enumerate}

Our first basic insight is that the projection of a point $q\in \PP^3$ under a pinhole camera $A : \PP^3 \dashrightarrow \PP^2$ may be viewed as the projection of a point $\vectorize (A) \in \PP^{11}$ under what we call a ``hypercamera" $Q : \PP^{11} \dashrightarrow \PP^2. $
This is reminiscent, and in fact a generalization, of a well-studied principle in computer vision known as \emph{Carlsson-Weinshall duality}~\cite{DBLP:journals/ijcv/CarlssonW98}.
This is the subject of~\Cref{sec:CWduality}.
Our~\Cref{thm:cw-main-thm} develops a reduced analogue of the ``atlas" of algebraic varieties proposed in~\cite{agarwal2022atlas}.
This addresses~\cite[\S 8.2, Q2]{agarwal2022atlas}.
The \emph{reduced joint image} and its dual, recently studied by Trager, Ponce, and Hebert, are two members of this atlas.
Carlsson-Weinshall duality amounts to a simple linear isomorphism between these two varieties.
In~\Cref{ex:hypersurface-case-1} and~\Cref{subsec:algebraic-consequences}, we explain how our perspective unifies previous approaches to resectioning constraints in the computer vision literature~\cite{DBLP:journals/ijcv/CarlssonW98, DBLP:conf/cvpr/TragerHP19,quan,DBLP:conf/eccv/SchaffalitzkyZHT00}, which can all be obtained from the ideal $I(\Gammaqbarp^{m,n})$ by specialization.

\begin{figure}[h]
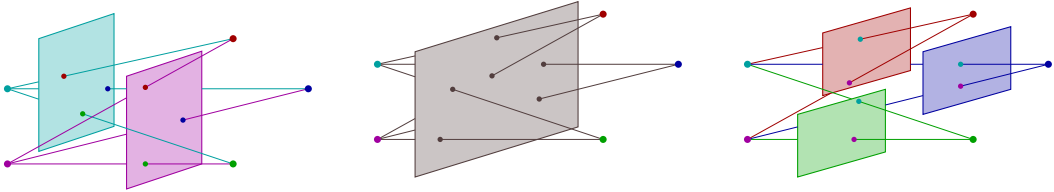

    \centering
    \include{CW-duality-illustration}
    \caption{Two reduced cameras viewing three 3D points (left) are Carlsson-Weinshall dual to three reduced cameras viewing two 3D points (right). See~\Cref{subsec:geometric-formulation} for details.}
    \label{fig:cw-duality-illustration}
\end{figure}

\Cref{thm:reduced-scheme-theoretic} in~\Cref{subsec:algebraic-consequences} shows that reduced resectioning varieties for generic point configurations are scheme-theoretically cut out by bilinear forms.
This stands in stark contrast to the high degree polynomials in~\Cref{thm:main-ideals}, whose proof we complete in~\Cref{sec:ideals}.
Finally, in~\Cref{sec:resectioning}, we address~\cite[\S 8.1, Q6]{agarwal2022atlas} by investigating the \emph{Euclidean distance degree} of the resectioning variety in affine pixel coordinates.
This is a number that quantifies the algebraic complexity of a natural Euclidean distance optimization formulation of the camera resectioning problem.
Our main contribution, based on evidence supplied by computational experiments, is~\Cref{conjecture:ED}, giving a formula for this quantity as a cubic polynomial in $n.$
The statement is analogous to, and inspired by, the \emph{multiview conjecture}, recently resolved by Maxim, Rodriguez, and Wang~\cite{MRW20}.
We conclude with a short discussion in~\Cref{sec:conclusion}.

\subsection{Notation and conventions}\label{subsec:notation}

Our notation largely follows that established in~\cite{agarwal2022atlas}.
Our basic algebro-geometric objects are affine and projective varieties over the field of complex numbers $\CC .$ The symbol $\PP^n$ denotes complex $n$-dimensional projective space, which we may also identify with the projectivization $\PP (V)$ of any $(n+1)$-dimensional complex vector space $V$.
As in the introduction, known quantities will usually be designated with a bar $\bar{\bullet}.$
This bar is also used to denote the \emph{Zariski closure} of a set: its usage will be clear from the context.
If we wish to emphasize that given quantities in certain scenarios may be ``noisy" due to deviations from the pinhole model or erroneous measurements, we instead use $\widetilde{\bullet }.$

\section{Resectioning vs Triangulation}\label{sec:res-triang}

Let us recall a ``universal" version of the imaging map~\eqref{eq:cam}.
This is the map which sends $m$ cameras $A_1, \ldots , A_m \in \PP \left( \operatorname{Hom}_\CC  (\CC^4, \CC^3) \right) \cong \PP^{11}$ and $n$ points $q_1, \ldots , q_n \in \PP^3$ to $mn$ points in $\PP^2.$
The graph of this rational map is an incidence correspondence, dubbed the \emph{image formation correspondence} in~\cite{agarwal2022atlas},
\begin{equation}\label{eq:image-formation-correspondence}
\GammaAqp^{m,n} = \overline{\{
(\bA, \qq, \pp )
\in (\PP^{11})^m \times (\PP^3)^n \times (\PP^2)^{mn}
\mid 
A_i q_j \sim p_{i j} \quad \forall i \in [m], \, j \in [n]
\}}.
\end{equation}
Given a generic camera arrangement $\bar{\bA} = (\bar{A}_1, \ldots , \bar{A}_m) \in \left( \PP^{11} \right)^m,$ one may also consider the associated \emph{multiview variety}.
In the notation of~\cite{agarwal2022atlas}, this may be defined as
\begin{equation}\label{eq:multiview-variety}
\GammaAbarp^{m,n} = \{
\pp 
\in (\PP^2)^{mn}
\mid
(\bar{\bA}, \qq , \pp ) \in \GammaAqp^{m,n}
\text{ for some }
\qq \in \left( \PP^3 \right)^n  
\}.
\end{equation}
Multiview varieties and their vanishing ideals are well-understood objects.
Our present study of camera resectioning is based on the following definition, which parallels~\eqref{eq:multiview-variety} in that the role of cameras and 3D points are switched.
\begin{definition}\label{defn:resectioning}
The $m$-camera \emph{resectioning variety} associated to a given point arrangement $\bar{\qq}\in \left(\PP^3 \right)^n$ is the multiprojective variety \begin{equation}\label{eq:resect-definition}
\Gammaqbarp^{m,n} = \left\{ \pp \in  \left( \PP^2 \right)^{mn} \mid (\bA, \bar{\qq}, \pp) \in \GammaAqp^{m,n} \text{ for some } \bA \in \left( \PP^{11} \right)^m\right\}.
\end{equation}
The vanishing ideal $I(\Gamma_{\bar{\qq}, \pp}^{m,n})$ is the \emph{resectioning ideal} of $\bar{\qq}.$ 
\end{definition}

\begin{remark}\label{remark:six-points}
It turns out that $\Gamma_{\bar{\qq}, \pp}^{m,n} = (\PP^2)^{mn}$ if and only if $n < 6,$ assuming $\bar{\qq} \in \left( \PP^3 \right)^n$ is sufficiently generic.
Thus we assume $n\ge 6$ throughout this section.
\end{remark}
To better explain the analogy between resectioning and triangulation, we collect several previous results about the multiview ideals $I(\GammaAbarp^{m,n})$ in~\Cref{thm:multiview-omnibus} below.
Our first main result,~\Cref{thm:main-ideals}, involves certain multilinear \emph{focal polynomials} which belong to the resectioning ideal $I( \Gamma_{\bar{\qq}, \pp}^{m,n})$.
These are structurally very similar to the classically-known focal polynomials belonging to $I( \Gamma_{\bar{\bA}, \pp}^{m,n})$.
We briefly recall a derivation of these constraints.
Suppose we are given a camera arrangement 
$\bar{\bA} \in \left( \PP^{11}\right)^n$. Consider a generic point 
\[
(\bar{\bA}, \qq, \pp) = (\bar{A}_1, \ldots , \bar{A}_m , q_1, \ldots , q_n, p_{1 1}, \ldots , p_{m n}) \in \GammaAqp^{m,n}.
\]
Fixing representatives for this point in homogeneous coordinates, there exist nonzero scalars $\lambda_{1 1} , \ldots , \lambda_{m n} \in \CC$ which satisfy the equations
\begin{equation}\label{eq:multiview-in-iamge-R-lambda}
\bar{A}_i q_j = \lambda_{i j} p_{i j} , \quad 1 \le i \le m, \, 1 \le j \le n.
\end{equation}
From these conditions, one may obtain certain multilinear polynomials in $\bar{A}_i, p_{i j}$ alone, known in various sources as $k$-focals or $k$-multilinearities.
Specifically, for each $j=1, \ldots , n$ and any subset $\sigma = \{ \sigma_1, \ldots , \sigma_k \} \subset [m]$ of size $\ge 2$, the matrix
\begin{equation}\label{eq:A-foca}
\begin{bmatrix} \bar{A}_{\sigma_1} & p_{\sigma_1} & \cdots & 0\\
\vdots & \vdots & \ddots & \vdots \\
\bar A_{\sigma_k} & 0 & \cdots &  p_{\sigma_k} 
\end{bmatrix}
\end{equation}
must be rank-deficient.
The maximal $(4+ k) \times (4+ k)$ minors of these matrices are the \emph{$k$-focals} associated with the camera arrangement $\bar{\bA}.$

In~\Cref{thm:multiview-omnibus}, we collect several previous results which make the relationship between $\GammaAbarqp^{m,n}$ and the $k$-focals more precise.
These results impose progressively stronger genericity assumptions on the camera arrangement $\bar{\bA}.$
\begin{theorem}\label{thm:multiview-omnibus} 
Let $\bar{\bA} = (\bar{A}_1,\ldots , \bar{A}_m)$, for $m\ge 2,$ be a fixed camera arrangement.
\begin{enumerate}
    \item \cite[Theorem 2.1]{Hilb} If all maximal $4\times 4$ minors of the matrix $\left[ \hspace{-.4em} \begin{array}{c|c|c}
\bar{A}_1^T & \cdots & \bar{A}_m^T
    \end{array} \hspace{-.4em}\right]$ are nonzero,
    then the $k$-focals for $k\in \{ 2,3,4 \}$ form a universal Gr\"{o}bner basis for $I(\GammaAbarp^{m,n}).$
    \item \cite[Theorem 3.7]{idealMultiview} If $\bar{\bA}$ is such that the camera centers are distinct, then the $k$-focals for $k \in \{ 2, 3 \}$ generate the vanishing ideal $I (\Gammaqbarp^{m,n}).$
    \item \cite[Theorem 5.6]{idealMultiview} If $\bar{\bA}$ is such that the camera centers are distinct and do not lie in a common plane, then the $2$-focals determine $\Gammaqbarp^{m,n}$ as a subscheme of $\left(\PP^2 \right)^{mn}$.
\end{enumerate}
\end{theorem}
Turning now to camera resectioning, suppose we are instead given $\bar{\qq} \in \left( \PP^3\right)^n$.
Similar to~\eqref{eq:multiview-in-iamge-R-lambda}, we wish to obtain conditions involving only $\bar{q}_j$ and $p_{i j}$ from 
\begin{equation}\label{eq:in-iamge-R-lambda}
A_i \bar{q}_j = \lambda_{i j} p_{i j} , \quad 1 \le i \le m, \, 1 \le j \le n.
\end{equation}
To obtain these conditions, we may apply a well-known identity involving the matrix Kronecker product, denoted $\otimes $, and the vectorization operator $\vectorize (\bullet)$, which stacks the columns of a matrix vertically.

\begin{proposition}[See eg.~{\cite[p252, Exercise 22]{horn-topics}}]\label{prop:horn}
For any $M \in \CC^{q \times r},$ $N \in \CC^{r \times s}$,
\begin{equation}\label{eq:amazing-identity}
\vectorize \left(M  N\right) = \left(I_{s \times s} \otimes M\right) \vectorize \left(N\right),
\end{equation}
where $I_{s \times s}\in \CC^{s \times s}$ is the identity matrix.
\end{proposition}
We apply this identity with $M = \bar{q}_j^\top$ and $N = A_i^\top$.
For the $3\times 12$ matrix $I_{3 \times 3} \otimes \bar{q}_j^\top$, we introduce the notation 
\begin{equation}\label{eq:q-check}
\bar{Q}_j := I_{3 \times 3} \otimes \bar{q}_j^\top = 
\begin{bmatrix}
\bar{q}_j^\top & 0 & 0 \\
0 & \bar{q}_j^\top &  0\\
0 & 0 & \bar{q}_j^\top 
\end{bmatrix}.
\end{equation}
Combining~\eqref{eq:in-iamge-R-lambda} and~\Cref{prop:horn}, we deduce that 
\[
\bar{Q}_j \vectorize \left(A_i^\top \right) = \lambda_{i j} p_{i j}, \quad 1 \le i \le m, \, 1 \le j \le n.
\]
Equivalently, for each $i=1, \ldots , m$ we have
\[
\begin{bmatrix}
\bar Q_1 & p_{i 1} & \cdots & 0\\
\vdots & \ddots & \vdots \\
\bar Q_n & 0& \cdots & p_{i n} 
\end{bmatrix} 
\begin{bmatrix}
\vectorize (A_i^\top) \\
-\lambda_{i 1} \\
\vdots \\ 
-\lambda_{i n}
\end{bmatrix}
=
\begin{bmatrix}
0\\
\vdots \\ 
0
\end{bmatrix}.
\]
Thus, if $\pp \in \Gammaqbarp^{m,n}$, then we have the  rank constraints
\begin{equation}\label{eq:rank-constraint}
\rank \begin{bmatrix}
\bar Q_1& p_{i 1} & \cdots & 0\\
\vdots & \vdots & \ddots & \vdots \\
\bar Q_n & 0 & \cdots & p_{i n} 
\end{bmatrix} < 12 + n.
\end{equation}
We observe that this rank constraint is equivalent to the vanishing of all maximal $(12 + n) \times (12 + n)$ minors.
These minors are homogeneous polynomials in the entries of each $\bar Q_i$ and $p_{i j}$; indeed, for any nonzero scalars $c_1, \ldots , c_n, c_1' , \ldots , c_n',$
\begin{align}
\rank 
\begin{bmatrix} c_1 \bar Q_1& c_1 ' p_{i 1} & \cdots & 0\\
\vdots & \vdots & \ddots & \vdots \\
c_n \bar Q_n & 0 & \cdots & c_n ' p_{i n} 
\end{bmatrix} 
&=
\rank 
\begin{bmatrix} \bar Q_1& c_1^{-1} p_{i 1} & \cdots & 0\\
\vdots & \vdots & \ddots & \vdots \\
\bar Q_n & 0 & \cdots & c_n^{-1} p_{i n} 
\end{bmatrix} \nonumber \\
&=
\rank 
\begin{bmatrix} \bar Q_1& p_{i 1} & \cdots & 0\\
\vdots & \vdots & \ddots & \vdots \\
\bar Q_n & 0 & \cdots & p_{i n} \label{eq:rescale}
\end{bmatrix} .
\end{align}
One may of course consider such rank constraints not only for $3\times 12$ matrices of the form~\eqref{eq:q-check}, but for any given arrangement of surjective linear maps, \[
\bar B_j : \PP^{11} \dashrightarrow \PP^2, \quad j=1, \ldots , n,
\]
represented by generic $3\times 12$ matrices.
To prevent confusion with cameras $A_i,$ we refer to each $\bar{B}_j$ as a \emph{hypercamera}.
We denote a general arrangement of hypercameras by $\bar{\bB} = (\bar{B}_1, \ldots , \bar{B}_n) \in \left( \PP^{35} \right)^n$. 
However, we instead write $\bar{\bQ}$ to denote
the special hypercamera arrangement associated to a point arrangement $\bar{\qq} \in \left( \PP^3 \right)^n$ by the rule~\eqref{eq:q-check}.

Let us also note that rank constraints analogous to~\eqref{eq:rank-constraint} hold for any subset of at least $6$ world points and their corresponding images.
This motivates the following definition, as well as the statement of our first result.

\begin{definition}\label{def:Ifoc}
Fix a hypercamera arrangement $\bar{\bB} = (\bar{B}_1, \ldots , \bar{B}_n)\in \left(\PP^{35}\right)^n$.
For any set $\{ \si_1, \ldots , \si_k \} \subset [n]$ of size $k\ge 6$ and an index $i\in [m]$, a \emph{$k$-focal} polynomial is any maximal $(12+k ) \times (12+k)$ minor of the $3k \times (12 + k)$ matrix 
\begin{equation}\label{eq:k-focal-matrix}
\begin{bmatrix}
\bar{B}_{\si_1} & p_{i \si_1} & \cdots & 0\\
\vdots & \vdots & \ddots & \vdots \\
\bar{B}_{\si_k} & 0 & \cdots & p_{i \si_k} 
\end{bmatrix} .
\end{equation}
From context, it will be clear whether ``focals" refers to the polynomials in~\Cref{def:Ifoc} or their triangulation counterparts.
The ideal in $\CC [\pp ]$ generated by all $k$-focals, $6\leq k\leq m$, is the $m$-camera \emph{focal ideal} $\Ifoc{m} (\bar{\bB}).$ 
For a given point arrangement $\qqbar$, we define its focal ideal $\Ifoc{m} (\bar{\qq})$ to be the focal ideal $\Ifoc{m} (\bar{\bQ})$ for the associated hypercamera arrangement $\Qbar$.
\end{definition}

\begin{theorem}\label{thm:main-ideals}
Let $m,n \ge 1$ be integers. 
For any point arrangement $\qqbar \in (\PP^3)^n$ such that no four points are coplanar, we have
\[
I (\Gammaqbarp^{m,n}) = \Ifoc{m} (\bar{\qq} ),
\]
and the set of all $k$-focals for $6\le k \le 12$ forms a universal Gr\"{o}bner basis for this ideal.
\end{theorem}
\Cref{thm:main-ideals} is the resectioning analogue of~\Cref{thm:multiview-omnibus} part (1).
Directly adapting the proof of this result is not straightforward.
This is because $\bar{\bQ}$ is a very special hypercamera arrangement.
Nevertheless, the noncoplanarity hypothesis in~\Cref{thm:main-ideals} ensures that $\bar{\bQ}$ is generic enough for Gr\"{o}bner basis arguments to be applied.

In the setting of triangulation, we note that the range of interesting focals $2 \le k \le 4$ is much smaller than in~\Cref{thm:main-ideals}, and in this setting the $k$-focals correspond to well-understood objects in multiview geometry---namely, fundamental matrices, trifocal tensors, and quadrifocal tensors~\cite[cf.~Ch.~17]{HZ04}.
It would seem that the $k$-focals for resectioning are less well-understood.
Nevertheless, in~\Cref{ex:hypersurface-case-1},~\Cref{subsec:algebraic-consequences}, we observe that they do specialize to ``dual" multiview constraints appearing in the literature.

As a warm-up, we establish a set-theoretic variant of~\Cref{thm:main-ideals}.
By analogy with~\eqref{eq:multiview-variety}, let us define for any hypercamera arrangement $\bar{\bB} \in \left(\mathbb{P} ^{35}\right)^{n}$ the variety $\Gamma_{\bar{\bB}, \pp}^{n,m}$ to be the closed image of the associated imaging map $\left( \mathbb{P}^{11} \right)^m \dashrightarrow \left( \PP^2 \right)^{mn}.$
In other words, $\Gamma_{\bar{\bB}, \pp}^{n,m}$ is a hypercamera version of the multiview variety.
When $\bar{\bB}=\Qbar,$ we have the following result.
\begin{proposition}\label{prop:set-theoretic}
Fix $\bar{\qq} = (\bar{q}_1, \dots, \bar{q}_n) \in (\PP^3)^n$ with no four $\bar{q}_i$ coplanar. Then
\[
\Gamma_{\Qbar , \pp}^{n,m} = \Gammaqbarp^{m,n} =  \operatorname{V}(\Ifoc{m} (\qqbar )).
\]
\end{proposition}
\begin{proof}
It is relatively straightforward to prove the inclusions
\[
\Gamma_{\Qbar , \pp}^{n,m} \subset  \Gammaqbarp^{m,n} \subset  \operatorname{V}(\Ifoc{m} (\qqbar )),
\]
so we focus on the harder inclusion $\operatorname{V}(\Ifoc{m} (\qqbar )) \subset \Gamma_{\Qbar , \pp}^{n,m}$.
This is also where we need the noncoplanarity assumption.
Consider any point 
\[\pp \in \operatorname{V}(\Ifoc{m} (\qqbar )).\]
We will construct a sequence of points $(\pp^{(k)}) \in \Gamma_{\Qbar, \pp}^{n,m}$ converging to $\pp .$
To simplify notation in what follows, we consider the case $m=1.$
When $m>1,$ the same construction applies component-wise.
We write $p_i$ in place of $p_{1 i},$ so that the kernel of the matrix 
\[
\begin{bmatrix}
    \bar Q_1 & p_{1} & \cdots & 0\\
    \vdots & \ddots & \vdots \\
    \bar Q_n & 0& \cdots & p_{n} 
    \end{bmatrix} 
    \]
contains a point $v = 
    [ v_{1} : \cdots : v_{12+n}] \in \PP^{11 + n}$.
    Let us fix homogeneous coordinates for $p_{1}, \ldots , p_{n}, \bar{q}_1, \ldots , \bar{q}_n, v$.
    We define
    \[A = \begin{bmatrix}
    v_{1} & \cdots & v_{4} \\
    \vdots & \ddots & \vdots \\
    v_{9} & \cdots & v_{12}
    \end{bmatrix},
    \quad
    \lambda_{j} = -v_{12+j}.
    \]
Let us first observe that the matrix $A$ is nonzero, for otherwise we would have 
\[
\lambda_j p_j = A \bar{q}_j = \bar{Q}_j \vectorize (A^\top) = 0 
\quad \Rightarrow \quad 
\lambda_j = 0
\]
for all $j$, contradicting the fact that $v\ne 0.$
Next, observe that at most three of the $\lambda_j$ can be zero: otherwise, four of the points $\bar{q}_j$ would lie in some plane containing the kernel of $A$, contradicting our hypothesis that $\bar{\qq}$ is noncoplanar.
It follows that we can find a nonzero $3\times 4$ matrix $A'$ with $A' \bar{q}_j = p_j$ for each $j$ with $\lambda_j=0$.
Fix such a matrix $A'.$
We now construct the desired sequence $(\pp^{(k)})_{k\ge 1}$.
Set
\begin{align*}
A^{(k) } &= A + (1/k) A', \\
p_j^{(k)} &= A^{(k)} \bar{q}_j, \quad j=1, \ldots , n.
\end{align*}
To show convergence, note that when $\lambda_j = 0$ we have
\[
p_j^{(k)} = (1/k) A' \bar{q}_j \sim A' \bar{q}_j = p_j.
\]
For $\lambda_j \ne 0$ we attain $p_j$ in the limit as $k\to \infty $, since
\[
p_j^{(k)} = A \bar{q}_j + (1/k) A' \bar{q}_j \to A \bar{q}_j \sim p_j.
\]
In summary, we have found for any point $\pp \in \operatorname{V}(\Ifoc{m} (\qqbar ))$ a sequence of points $(\pp^{(k)}) \in \Gamma_{\Qbar, \pp}^{n,m}$ converging to $\pp .$
Since $\Gamma_{\Qbar, \pp}^{n,m}$ is closed in the Euclidean topology, we deduce the needed inclusion: $\operatorname{V}(\Ifoc{m} (\qqbar )) \subset \Gamma_{\Qbar , \pp}^{n,m}$.
\end{proof}
We conclude this section with the simplest interesting example of a resectioning variety.
\begin{example}\label{ex:hypersurface-case-1}
For $m=1$ camera and a generic point arrangement $\bar{\qq} = (\bar{q}_1, \ldots , \bar{q}_6) \in \left( \PP^3 \right)^n,$ the resectioning variety $\Gammaqbarp^{1,6}$ is a hypersurface in $\left( \PP^{2} \right)^6.$
Applying a suitable permutation to the rows of the $6$-focal matrix, this hypersurface has an $18 \times 18$ determinantal representation,
\begin{equation}\label{eq:det-hypersurface}
\det 
\begin{bmatrix}
\bar{q}_1^T & & & p_1 [1] & \\
\vdots  & & & & \ddots & \\
\bar{q}_6^T & & & & &  p_6 [1] \\
& \bar{q}_1^T& & p_1 [2] & \\
  & \vdots & & & \ddots & \\
& \bar{q}_6^T & & & &  p_6 [2] \\
& & \bar{q}_1^T &  p_1 [3] & \\
& & \vdots   & & \ddots & \\
& & \bar{q}_6^T  & & &  p_6 [3] 
\end{bmatrix}
=0.
\end{equation}
Out of the $3^6=729$ possible terms of a sextilinear form on $\left( \PP^2 \right)^6$, the special structure of the $6$-focal determinant dictates that only $\binom{6}{2} \binom{4}{2} = 90$ can be nonzero.
On the other hand,~\Cref{lemma:make-minor-generic} below shows that applying a general linear change of coordinates to~\eqref{eq:det-hypersurface} has the effect that all $729$ possible terms become nonzero.
This highlights an important distinction between resectioning and multiview ideals---the initial ideal for generic data $\qqbar $ is not the same as the $\mathbb{Z}^6$-graded Borel-fixed generic initial ideal (cf.~\cite[\S 3]{Hilb}.)
Letting $<$ denote the lexicographic order with $p_6[3] < p_6 [2] < p_6 [1] < p_5 [3] < \cdots <p_1[1],$ we have
\begin{equation}\label{eq:in-not-gin}
\begin{split}
\operatorname{in}_< ( I (\Gamma_{\bar{\qq}, \pp}^{1,6} ) = \langle 
p_{1} [1] p_{2} [1] p_{3} [2] p_{4} [2] p_{5} [3] p_{6} [3]
\rangle , \\
\operatorname{gin}_< ( I (\Gamma_{\bar{\qq}, \pp}^{1,6} ) = \langle p_{1} [1] p_{2} [1] p_{3} [1] p_{4} [1] p_{5} [1] p_{6} [1] \rangle .
\end{split}
\end{equation}
Interestingly,~\eqref{eq:det-hypersurface} also has several smaller determinantal representations.
Many of these may be obtained from~\eqref{eq:det-hypersurface} using Schur complements.
For example, we have the $12\times 12$ determinantal representation
\begin{equation}\label{eq:12-det}
\left( p_1 [3] \cdots p_6[3] \right)^{-1}
\det 
\begin{bmatrix}
p_1[3]\bar{q}_1^T & & - p_1 [1] \bar{q}_1^T \\
\vdots  &  & \vdots & \\
p_6[3] \bar{q}_6^T & & - p_6 [1] \bar{q}_6^T \\
 & p_1 [3]\bar{q}_1^T & - p_1 [2] \bar{q}_1^T \\
& \vdots  & \vdots & \\
& p_6[3]\bar{q}_6^T & - p_6 [2] \bar{q}_6^T 
\end{bmatrix}
=0.    
\end{equation}
The $6$-focal determinant also has a $6\times 6$ determinantal representation, which specializes to~\eqref{eq:dual-fundamental-1} below after fixing  $\bar{q}_1, \ldots , \bar{q}_4, p_1, \ldots , p_4.$
This is the classical form of the constraint appearing in works such as~\cite{quan,DBLP:journals/ijcv/CarlssonW98}.
Finally, we note that Schaffilitzky et al.~\cite{DBLP:conf/eccv/SchaffalitzkyZHT00} derive a $3\times 3$ determinantal constraint relating 3D points and their 2D projections that is linear in a distinguished image point $p_6$.~\Cref{thm:main-ideals} implies that their determinant is a multiple of the $6$-focal determinant.
Notably, Schaffilitzky et al.~use their constraint to solve the minimal problem of reconstructing $6$ points from $3$ views.
Earlier works, eg.~\cite{DBLP:journals/ijcv/CarlssonW98}, had already observed that this problem is equivalent to the classical $7$ point problem in $2$ views.
This equivalence follows from the principle of Carlsson-Weinshall duality, which we revisit in the next section. 
\end{example}

\section{Carlsson-Weinshall duality revisited}~\label{sec:CWduality}

Recall the image formation variety $\GammaAqp^{m,n}$  from~\eqref{eq:image-formation-correspondence}.
Previous work of Agarwal et al.~\cite{agarwal2022atlas} explains how the problems of reconstruction, triangulation, and resectioning may all be understood in terms of slicing and projection operations on this variety. 
The relationships between the varieties produced by these operations are summarized in a diagram designated as an \emph{atlas} for the pinhole camera. 
One striking feature of the atlas's appearance is the apparently symmetric roles of cameras in $\PP^{11}$ and world points in $\PP^3.$
A simple explanation for this phenomenon is as follows: for a given camera center $c\in \PP^{3},$ world point $q \in \PP^{3},$ and image plane $L \in \Gr (\PP^2, \PP^3),$ we obtain the same projected point on $L$ whether we project $c$ through $q$ or project $q$ through $c$.
If we want to express this symmetry in terms of camera matrices instead of camera centers, one approach is to introduce coordinates on the image plane.
Indeed, there are an additional $\dim \PGL_3 = 8 = 11 - 3$ degrees of freedom in choosing projective coordinates on $L$.
A particular choice of coordinates leads directly to the framework of \emph{Carlsson-Weinshall (CW) duality} from the multiview geometry literature.

In this section, we point out that several world-to-image point constraints which were previously discovered using CW duality arise naturally as specializations of our focal constraints. We also show in~\Cref{thm:cw-main-thm} that Carlsson-Weinshall duality gives rise to a rational quotient of the image formation correspondence, and develop a \emph{reduced} version of the atlas that better explains the symmetry between cameras and world points---see~\Cref{fig:reduced-atlas}.

A direct application of the focal constraints described in~\Cref{sec:res-triang}
arises naturally in the setting of Carlsson-Weinshall (CW) duality. 
In the eponymous authors' celebrated work, CW duality is described as the notion that 
``\emph{problems of [resectioning] and [triangulation] from image data are... dual in the sense that they can be solved with the same algorithm depending on the number of [world] points and cameras}''~\cite{DBLP:journals/ijcv/CarlssonW98}.

In this section, we develop CW duality in the context of a \emph{reduced atlas}, analogous to that of \cite{agarwal2022atlas}, which makes the symmetry between cameras and 3D points evident.
\Cref{thm:cw-main-thm,thm:reduced-scheme-theoretic} explain how nodes in this atlas arise as rational quotients of their non-reduced counterparts.
The latter result also includes an analogue of~\Cref{thm:main-ideals}: the reduced resectioning variety is cut out scheme-theoretically by bilinear forms for a sufficiently generic point configuration.

\begin{remark}\label{rem:coordinate-free}
Recent work by Trager, Hebert, and Ponce~\cite{DBLP:conf/cvpr/TragerHP19} demonstrates that the exact coordinates of the camera centers and world points are not essential features of CW duality, contrary to the original setup.
For simplicity, we state the main results of this section with respect to the conventional projective frame defined in~\eqref{eq:standard-position}.
\end{remark}

\subsection{Geometric formulation}\label{subsec:geometric-formulation}
For $m$ cameras and $n$ world points, we define the \emph{reduced image formation correspondence} to be the variety
\begin{equation}\label{eq:reduced-image-formation-correspondence}
\Rhoaqp^{m,n} = 
\overline{\{
(\ba, \qq, \pp )
\in (\PP^3)^m \times (\PP^3)^n \times (\PP^2)^{mn}
\mid 
A(a_i) \cdot q_j \sim p_{i j} \quad \forall i \in [m], \, j \in [n]
\}},
\end{equation}
where for $a_i = [a_{i 1} : a_{i 2} : a_{i 3} : a_{i 4} ] \in \PP^3$ we define
\begin{equation}\label{eq:Aq}
A( a_i ) = 
\begin{bmatrix}
a_{i 1} & 0 & 0 & a_{i 4}\\
0 & a_{i 2} & 0 & a_{i 4}\\
0 & 0 & a_{i 3} & a_{i 4}
\end{bmatrix}.
\end{equation}
When $A(a_i)$ is of full rank, we call it the \emph{reduced camera matrix} associated to the point $a_i.$
The center of a reduced camera matrix $A(a_i)$ is $\mathcal C(a_i)$, where $\mathcal C$ is the quadratic Cremona involution
\begin{align}
\mathcal{C} : \PP^3 &\dashrightarrow \PP^3 \nonumber\\
[a_1 : a_2 : a_3 : a_4 ] &\mapsto [1/a_1 : 1/a_2 : 1/a_3 : - 1 / a_4]. \label{eq:cremona}
\end{align}
Note that $\mathcal C(a_i)$ is defined exactly when at most one $a_{ij}$ is zero, or equivalently, when $A(a_i)$ is a full-rank camera matrix with a well-defined center.

The key observation of Carlsson-Weinshall duality is expressed by the symmetric roles of a 3D point $q_j$ and a reduced camera $A(a_i)$ in image formation:
\begin{equation}\label{eq:cw-flip}
A (a_i) q_j = A(q_j) a_i \quad \forall i=1, \ldots, m, \, \, j=1, \ldots , n.
\end{equation}
The special form of the reduced camera matrix arises from fixing a projective basis in each image and a partial projective basis in the world. 
We adopt the notation of~\cite[Ch.~16]{HZ04}:
\begin{align}
E_1 = [1 : 0: 0 : 0], &\quad e_1 = [1: 0 : 0], \nonumber \\
E_2 = [0 : 1: 0 : 0], &\quad e_2 = [0: 1 : 0], \nonumber \\
E_3 = [0 : 0: 1 : 0], &\quad e_3 = [0: 0 : 1], \nonumber \\
E_4 = [0 : 0: 0 : 1], &\quad e_4 = [1: 1 : 1], \nonumber \\
E_5 = [1 : 1: 1 : 1]. \label{eq:standard-position}
\end{align}
Each set of four points $E_1, \ldots E_4\in \PP^3$, $e_1,\ldots , e_4\in \PP^3$ is said to span a \emph{reference tetrahedron} in $\PP^3.$
The geometry relating these points and the Cremona transformation $\mathcal{C}$ can be appreciated in~\Cref{fig:twisted-cubic}.

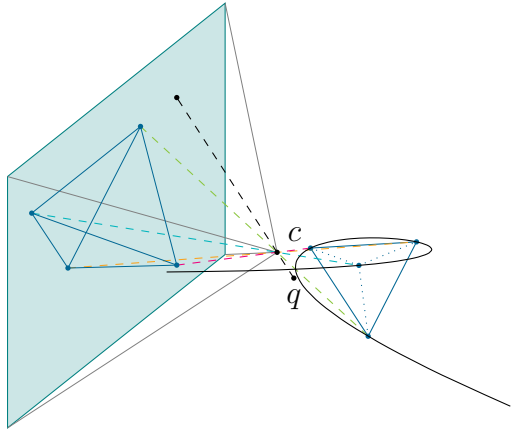
\begin{figure}[h]
\tdplotsetmaincoords{48}{220}
\centering
\begin{tikzpicture}[scale = 1, tdplot_main_coords]
    \filldraw[draw = teal, fill = teal!20] 
    (5, 2.5,1)
    -- (5, -2,1)
    -- (5, -2,-3.5)
    -- (5, 2.5,-3.5)
    -- cycle;
    \draw[very thin, gray] (2, 1/2,0) -- (5, 2.5,1); 
    \draw[very thin, gray] (2, 1/2,0) --  (5, -2,1);
    \draw[very thin, gray] (2, 1/2,0) -- (5, -2,-3.5);
    \draw[very thin, gray] (2, 1/2,0) -- (5, 2.5,-3.5);
    \draw[thin, MidnightBlue] (0,1,0) -- (1,1,1) -- (0,0,1) -- cycle; 
    \draw[MidnightBlue, dotted] (0,1,0) -- (1,0,0);
    \draw[MidnightBlue, dotted] (1,0,0) -- (1,1,1);
    \draw[MidnightBlue, dotted] (1,0,0) -- (0,0,1);

    \draw[thin, dashed, Aquamarine] (5,2,0) -- (1,0,0); 
    \draw[thin, dashed, LimeGreen] (5, -1/4, 0) -- (0,1,0); 
    \draw[thin, dashed, YellowOrange] (5, 5/4, -3/2) -- (0,0,1); 
    \draw[thin, dashed, RubineRed] (5, -1, -3) -- (1,1,1); 

    \draw[very thin, MidnightBlue] (5,2,0) -- (5, -1/4, 0) -- (5, 5/4, -3/2) -- cycle;
    \draw[very thin, MidnightBlue] (5, -1, -3) -- (5,2,0);
    \draw[very thin, MidnightBlue] (5, -1, -3) -- (5, -1/4, 0);
    \draw[very thin, MidnightBlue] (5, -1, -3) -- (5, 5/4, -3/2);

    \fill (3/2, 3/4, 0) circle (1pt) node[anchor = north]{$q$}; 
    \fill (2, 1/2, 0) circle (1pt) node[anchor = south west]{$c$}; 
    \draw[thin, dashed] (3/2,3/4,0) -- (2, 1/2, 0);
    
    \fill[MidnightBlue] (1,0,0) circle (1pt) 
    ;
    \fill[MidnightBlue] (0,1,0) circle (1pt) 
    ;
    \fill[MidnightBlue] (0,0,1) circle (1pt) 
    ;
    \fill[MidnightBlue] (1,1,1) circle (1pt) 
    ;
    
    \fill[MidnightBlue] (5,2,0) circle (1pt) 
    ;
    \fill[MidnightBlue] (5, -1/4, 0) circle (1pt) 
    ;
    \fill[MidnightBlue] (5, 5/4, -3/2) circle (1pt) 
    ;
    \fill[MidnightBlue] (5, -1, -3) circle (1pt) 
    ;

    \fill (5, -1, 0) circle (1pt) node[anchor = south]{};

    \draw[dashed, thin] (2, 1/2, 0) -- (5, -1, 0);
\draw[domain=-.4:3.5,smooth,variable=\t]plot({(\t)*(\t-2)*(2*\t-5)*(1/3)},{(\t-3)*(\t-2)*(2*\t+1)*(1/6)},{(\t)*(\t-3)*(-1/2)});
\end{tikzpicture}
\caption{Four fixed points $E_1, \ldots , E_4 \in \PP^3$ determine a reference tetrahedron.
They project through the camera center $c\in \PP^3$ to the four points $e_1, \ldots e_4\in \PP^2.$
The Cremona transformation $\mathcal{C}$ maps the line through $c$ and $q$ to the unique twisted cubic passing through $\mathcal{C} (q), \mathcal{C} (c), E_1, \ldots , E_4.$ 
}\label{fig:twisted-cubic}
\end{figure}

Given a camera of the form $A = A(a), a \in \PP^3$, we have that $A(a) E_i = e_i$ for $i = 1, \dots, 4$. 
The converse is true as well; that is, a camera matrix takes the reduced form~\eqref{eq:Aq} if and only if it sends $E_i$ to $e_i$ for each $i=1, \ldots ,4.$
As we will soon demonstrate, there is a rational group action by an algebraic group $\calG_m$ on $\GammaAqp^{m,n+4}$ for which each $\calG_m$-orbit, where defined, contains a unique element of $\Rhoaqp^{m,n}$.
That is, $\Rhoaqp^{m,n}$ can be thought of as a kind of quotient of the general image formation correspondence. \Cref{thm:cw-main-thm} makes this precise using the notion of a \emph{rational quotient.}
In this reduced setting, the roles of camera centers and world points are manifestly symmetric: a point $(\ba, \qq, \pp) \in \Rhoaqp^{m,n}$ is also a point in $\Rhoaqp^{n,m}$ after swapping the $\ba$ and $\qq$ factors. 
By this observation, we then get the isomorphism $\Rhoaqp^{m,n} \simeq \Rhoaqp^{n,m}$. 

Just as a point in $\GammaAqp^{m,n}$ can be thought of as a configuration of cameras and points, a point in $\Rhoaqp^{m,n}$ can be thought of as such a configuration up to certain coordinate changes.
More precisely, points in $\Rhoaqp^{m,n}$ correspond to orbits in $\GammaAqp^{m,n}$ under the action of a group $\mathcal{G}_m$ consisting of coordinate changes in the world and each of the $m$ images.
Up to this group action, we may assume the image planes $L_1, \ldots , L_m \in \Gr (\PP^2, \PP^3)$ are all equal, ie.~$L_1=\ldots=L_m$. This explains the center image in~\Cref{fig:cw-duality-illustration}.

We now transition into a formal treatment of the notions described above.
Define 
$\calG_m = (\PGL_3)^m \times \operatorname{Stab}_{\, \PGL_4} \left(E_5 \right)$, 
an algebraic group of dimension $8m + 12$ which acts rationally on $\GammaAqp^{m,n}$ as follows:
\begin{align} 
\calG_m \times \GammaAqp^{m,n} &\dashrightarrow \GammaAqp^{m,n} \nonumber \\
(T_1, \ldots , T_m, S) &\cdot  (A_1, \ldots , A_m, q_1, \ldots q_n, p_{1 1}, \ldots , p_{m n} ) \label{eq:act-Gm}\\
&= (T_1 A_1 S^{-1}, \ldots , T_m A_m S^{-1}, S q_1 , \ldots , S q_n, T_1 p_{1 1}, \ldots , T_m p_{m n}). \nonumber 
\end{align}
To formalize the intuition that $\Rhoaqp^{m,n}$ is a quotient of $\GammaAqp^{m,n}$ by $\mathcal{G}_m$,  we recall the definition of a \emph{rational quotient} as follows.
\begin{definition}\label{cf:rational-quotient}
(cf.~\cite[\S 6.2]{Dolgachev}.) 
Let $X$ and $Y$ be irreducible algebraic varieties and $G$ an algebraic group acting rationally on $X.$
We say $Y$ is a \emph{rational quotient} of $X$ by $G$, and write $
\ratquo{X}{G}{Y}$, if $Y$ is a model for the field of $G$-invariant rational functions on $X$: that is, if there exists an isomorphism $\CC (Y) \cong \CC (X)^G.$
\end{definition}
A classical result due to Rosenlicht states that rational quotients always exist over any algebraically closed field (cf.~\cite[Theorem 6.2]{Dolgachev}.)
The following simple lemma provides sufficient conditions for recognizing a particular class of rational quotients in which the action yields a birational equivalence of $X$ with $G \times Y.$
\begin{lemma}\label{lemma:rat-quot}
Let $G$ be an algebraic group acting rationally on a variety $X.$
For a subvariety $Y \subset X$, 
we have $\ratquo{X}{G}{Y}$ if there exists a rational map
\begin{align*}
\mu_G : X &\dashrightarrow G \\
x &\mapsto \mu_G (x)
\end{align*}
such that $\mu_G (y) = \operatorname{id}_G$ for all $y$ in a dense open subset of $Y,$ and such that $\mu_G( g \cdot x) = \mu_G ( x) \, g^{-1}$ for all $(g,x)$ in dense open subset of $G\times X.$
Moreover, these assumptions imply that 
\begin{align*}
X &\dashrightarrow G \times Y\\
x &\mapsto (\mu_G (x), \mu_G (x) \cdot x) 
\end{align*}is a birational equivalence, with a rational inverse given by
\begin{align*}
 G \times Y &\dashrightarrow X\\
(g, y) &\mapsto g^{-1} \cdot y.
\end{align*}
\end{lemma}
\begin{proof}
A function $f\in \CC (Y)$ pulls back to a function $h \in \CC (X)$ defined by $h(x) = f(\mu_G (x) \cdot x) $ on $X$.
Our assumptions imply $h$ is $G$-invariant, since
\[
h(g\cdot x ) = f( \mu_G (g\cdot x)  \cdot (g\cdot x) ) 
= f\left( (\mu_G (x) g^{-1}) \cdot (g\cdot x) \right) = h(x).
\]
Let us write $\varphi^* : \CC (Y) \to \CC (X)^G.$
Since $Y \subset X,$ we also have the induced map $\iota^*: \CC (X)^G \to \CC (Y).$
We show that $\varphi^* $ and $\iota^* $ are mutual inverses.
Taking any function $f \in \CC (Y)$ and $y\in Y$ in its domain of definition, we calculate
\[
\iota^* \varphi^* f (y) = f ( \mu_G (y) \cdot y ) = f(y) .
\]
Similarly, for any fixed $f \in \CC (X)^G,$ the values $f(x)$ and $f( \mu_G (x) \cdot x)$ are defined for a dense open subset of $x \in X,$ for which we compute
\[
\varphi^* \iota^* f (x) = f ( \mu_G (x) \cdot x ) = f(x).
\]
This proves $\ratquo{X}{G}{Y}.$ 
The birational equivalence of $G \times Y$ and $X$ follows similarly.
\end{proof}

\begin{theorem}[CW duality]\label{thm:cw-main-thm}
For any $m,n \geq 0$, we have a birational equivalence of varieties
\[
\GammaAqp^{m,n+4} \cong_{\textrm{Bir}} 
\Rhoaqp^{m,n} \times 
\mathcal{G}_m ,
\]
which yields following commutative diagram (in which
each arrow labeled $\sim$ is a birational or biregular isomorphism.)
\[
\begin{tikzcd}[ampersand replacement=\&]
\GammaAqp^{m,n+4} \times (\PGL_3)^n \arrow[r, dashed, "\sim"] \arrow[d, dashed, "\rotatebox{90}{$\sim$}"] \& \GammaAqp^{n,m+4} \times (\PGL_3)^m \arrow[d, "\rotatebox{90}{$\sim$}", dashed]\\
(\Rhoaqp^{m,n} \times \calG_m) \times (\PGL_3)^n \ar[dashed, r, "\sim"] \ar[d] \& (\Rhoaqp^{n,m} \times \calG_n) \times (\PGL_3)^m \ar[d] \\
\Rhoaqp^{m,n} \ar[r, "\sim"] \& \Rhoaqp^{n,m}
\end{tikzcd}
\]
This diagram has the following 
additional properties:
\begin{enumerate}
\item If $\nu_{m,n}$ denotes any of the horizontal maps, we have $\nu_{n,m} \circ \nu_{m,n} = \operatorname{id}$ wherever both maps are defined.
\item The vertical maps express the reduced image formation variety as a rational quotient of the image formation correspondence, \begin{equation}\label{eq:rat-quot}
\ratquo{\GammaAqp^{m,n+4}}{\mathcal{G}_m}{\Rhoaqp^{m,n}}.
\end{equation}
\item The duality between the problems of exact resectioning and triangulation may be expressed in terms of this commutative diagram and certain projections: eg., for the bottom row, if $\pi_{\ba}', \pi_{\qq}'$ denote the projections from $\Rhoaqp$ that forget the $\ba$ and $\qq$ factors, then the diagram below commutes.
\begin{center}
\begin{tikzcd}[ampersand replacement=\&, column sep = small]
\& (\PP^3)^n \times (\PP^2)^{mn} \\
\Rhoaqp^{m,n} \ar[rr, "\sim"] \ar[rd, "\pi_{\qq}'"'] \ar[ru, "\pi_{\ba} '"] \&  \& \Rhoaqp^{n,m} \ar[ld, "\pi_{\ba}'"] \ar[lu, "\pi_{\qq}' "']\\
\& (\PP^3)^m \times (\PP^2)^{mn}
\end{tikzcd}
\end{center}
\end{enumerate}
\end{theorem}

\begin{proof}
We begin by constructing the maps that yield the rational quotient~\eqref{eq:rat-quot}.
This part follows by applying~\Cref{lemma:rat-quot} with $X = \GammaAqp^{m,n+4} $, $Y = \Rhoaqp^{m,n}$, and $G = \mathcal{G}_m.$
To obtain the inclusion $\Rhoaqp^{m,n} \subset \GammaAqp^{m,n+4}$ we define
\begin{align*}
\iota : \Rhoaqp^{m,n} &\to \GammaAqp^{m,n+4}\\
(a_1, \ldots , a_m, q_1, \ldots , q_n , p_{1 1}, \ldots , p_{m n}) &\mapsto \\
(A(a_1), \ldots , A(a_m), E_1, E_2, E_3, E_4, &\, q_1, \ldots , q_n , e_1, e_2, e_3, e_4, p_{1 1}, \ldots , p_{m n}).
\end{align*}
To construct the map $\mu_{\mathcal{G}_m} : \GammaAqp^{m,n+4} \dashrightarrow \mathcal{G}_m$, consider first the map
\begin{align*}
S : (\PP^3)^4 &\dashrightarrow \operatorname{Stab}_{\PGL_4}(E_5) \\
( q_1, \ldots , q_4) &\mapsto 
\bigg( 
\left[\hspace{-.2em}
\begin{array}{c|c|c|c}
q_1 & q_2 & q_3 & q_4
\end{array}\hspace{-.2em}
\right]
\cdot 
\diag \left( [5 2 3 4]_\qq, [1 5 3 4]_\qq, [1 2 5 4]_\qq , [1 2 3 5]_\qq \right)
\bigg)^{-1} ,
\end{align*}
where each $[5 2 3 4]_\qq, \ldots , [1 2 3 5]_\qq$ is the determinant of a matrix obtained by replacing $q_1, \ldots , q_4$ with $E_5$ in the $4\times 4$ matrix whose columns are $q_1, \ldots , q_4.$
We verify that $S(q_1,\ldots , q_4)$ is well-defined and contained in $\operatorname{Stab}_{\PGL_4}(E_5)$ using linear algebra. 
To ease notation,
we write $Q = \left[\hspace{-.2em}
\begin{array}{c|c|c|c}
q_1 & q_2 & q_3 & q_4
\end{array}\hspace{-.2em}
\right]$ and $D = \diag \left( [5 2 3 4]_\qq, [1 5 3 4]_\qq, [1 2 5 4]_\qq , [1 2 3 5]_\qq \right)$. 
Rescaling any of the $q_1, \ldots , q_4$ then rescales the matrix product $Q D$.
Using
Cramer's rule, we calculate that
\begin{align*}
S(q_1, \ldots , q_4) E_5 &= 
D^{-1} \cdot Q^{-1} E_5 \\
&= 
D^{-1} [
[5 2 3 4]_\qq : [1 5 3 4]_\qq : [1 2 5 4]_\qq : [1 2 3 5]_\qq ]\\
&= E_5.
\end{align*}
An analagous calculation can be used to verify that for any $S_0 \in \PGL_4$ we have
\begin{equation}\label{eq:T-act}
S( S_0 q_1, \ldots , S_0 q_4) = S(q_1, \ldots ,q_4) \, S_0^{-1}.    
\end{equation}
Similar to our definition of $S$ above, we may define a map
\begin{align*}
T : (\PP^2)^4 &\dashrightarrow \PGL_3 \\
( p_1, \ldots , p_4) &\mapsto 
\bigg( 
\left[\hspace{-.2em}
\begin{array}{c|c|c}
p_1 & p_2 & p_3
\end{array}\hspace{-.2em}
\right]
\cdot 
\diag \left( [4 2 3]_\pp, [1 4 3]_\pp, [1 2 4]_\pp \right) 
\bigg)^{-1},
\end{align*}
but we replace $p_1, \ldots ,p_3$ by $p_4$ (rather than $e_4$) when forming the expressions $[4 2 3]_\pp, \ldots , [1 2 4]_\pp.$
Once again, for $T_0 \in \PGL_3$ we have 
\begin{equation}\label{eq:S-act}
T( T_0 p_1, \ldots , T_0 p_4) = T(p_1, \ldots ,p_4) \, T_0^{-1}.    
\end{equation}
Finally, we define
\begin{align*}
\mu_{\mathcal{G}_m} : \GammaAqp^{m,n+4} &\dashrightarrow \mathcal{G}_n \\
(A_1, \ldots ,A_m, q_1, \ldots , q_{n+4} , p_{1 1} , \ldots , p_{m n}) &\mapsto \\
(T(p_{1 1}, &\ldots , p_{1 4}), \, \ldots ,\,  T(p_{1 1}, \ldots , p_{m 4}) , \, S(q)).
\end{align*}
We check that the two assumptions of~\Cref{lemma:rat-quot} are satisfied.
The map $\mu_{\mathcal{G}_m}$ fixes $\Rhoaqp^{m,n}$ pointwise since $T(e_1, e_2, e_3, e_4)$ and $S(E_1, E_2, E_3, E_4)$ both act as the identity.
Similarly, for sufficiently generic $g\in \mathcal{G}_m$ and $x\in \GammaAqp^{m,n+4}$ the assumption that $\mu_{\mathcal{G}_m} ( g\cdot x) = \mu_{\mathcal{G}_m} (x) \, g^{-1}$ follows from~\eqref{eq:T-act} and~\eqref{eq:S-act}.

Thus, we may conclude from~\Cref{lemma:rat-quot} that we have the rational quotient~\eqref{eq:rat-quot}, giving property 2 in the statement of the theorem. 
Moreover, the lemma implies that $\GammaAqp^{m,n+4} $ is birationally equivalent to $\Rhoaqp^{m,n} \times \mathcal{G}_m$, which allows us to define the vertical maps in the main diagram.
To complete the diagram, it suffices to define the bottom-most map, which is
\begin{align*}
\nu_{m,n} : \Rhoaqp^{m,n} &\to \Rhoaqp^{n,m} \\
(a_1,\ldots , a_m , q_1, \ldots , q_n, p_{1 1} , \ldots p_{m n}) &\mapsto (q_1,\ldots , q_n , a_1, \ldots , a_m, p_{1 1} , \ldots p_{n m}). 
\end{align*}
Now, to show that $\nu_{m,n}$ is an isomorphism, we use the symmetric equations~\eqref{eq:cw-flip}.
The remaining parts of the theorem now follow easily.
\end{proof}
\begin{figure}[h]
\centering 
\begin{tikzcd}[]
& & 
{\Rho_{\bar{\ba}, \qq,\pp}^{m,n} 
\arrow[dl] 
\arrow[from=dr, color=red]
} 
& 
{
\Rho_{\ba, \qq, \bar{\pp}}^{m,n} 
}
& 
{\Rho_{\ba, \bar{\qq},\pp}^{m,n} 
\arrow[dr] 
}
& & 
\\
&
{
\Rho_{\bar{\ba}, \pp}^{m,n}
} 
& 
\Rho_{\ba, \bar{\pp}}^{m,n} \arrow[from=ur, crossing over]
& 
{
\Rho_{\ba, \qq, \pp}^{m,n} 
\arrow[dr] 
\arrow[dl]
\arrow[u,color=red, crossing over]
\arrow[ur,color=red, crossing over]
}
&
{\Rho_{\qq,\bar{\pp}}^{m,n}}\arrow[from=ul, crossing over]
& 
{\Rho_{\bar{\qq}, \pp}^{m,n}}
&
\\ 
&
& 
{
\Rho_{\ba, \pp}^{m,n}
\arrow[ul,color=red]
\arrow[dr] 
\arrow[u,color=red]
} 
& 
& 
{
\Rho_{\qq, \pp}^{m,n}
\arrow[ur,color=red]
\arrow[dl] 
\arrow[u,color=red]
} 
& 
&
\\ 
&
& 
& 
{
\Rho_{ \pp}^{m,n}
} 
& 
&
&
\end{tikzcd}
\caption{An atlas for the reduced pinhole camera, cf.~\cite[Figure 1]{agarwal2022atlas}.
}\label{fig:reduced-atlas}
\end{figure}
The reduced image formation correspondence $\Rhoaqp^{m,n}$ sits at the center of the \emph{reduced atlas} depicted in~\Cref{fig:reduced-atlas}.
Following~\cite{agarwal2022atlas}, we may define the remaining entities in this figure using slices and projections of the reduced image formation correspondence. 
For instance, the varieties $\Rho_{\ba , \pp }^{m,n}$ and $\Rho_{\qq , \pp}^{m,n}$ are defined, respectively, as the image under the coordinate projections $\pi_{\qq  }' : \Rhoaqp^{m,n} \to \left( \PP^3 \right)^m \times \left( \PP^2 \right)^{mn}$, $\pi_{\ba} ' : \left( \PP^3 \right)^n \times \left( \PP^2 \right)^{mn}$ appearing in~\Cref{thm:cw-main-thm}. 
Slicing the variety $\Rho_{\qq , \pp }^{m,n}$ with the coordinate planes defined by $\qq = \bar{\qq},$ we obtain the \emph{reduced resectioning variety} $\Rho_{\bar{\qq}, \pp}^{m,n}.$ 
Up to Zariski closure, this is the \emph{dual reduced joint image} introduced in~\cite{DBLP:conf/cvpr/TragerHP19}.
\Cref{thm:cw-main-thm} above and~\Cref{thm:reduced-scheme-theoretic} below illustrate the correspondence between these varieties and their non-reduced counterparts in~\cite{agarwal2022atlas}, and provide an explanation for the vertical symmetry present in both atlases.
\subsection{Algebraic consequences}\label{subsec:algebraic-consequences}
Readers familiar with multiview geometry will no doubt wonder how the focal constraints of~\Cref{thm:main-ideals} relate to various ``dual multiview constraints", derived by Carlsson, Weinshall, and others.
All of these previously-studied constraints may be interpreted as polynomials vanishing on $\Rho_{\bar{\qq}, \pp}^{m,n}.$
Specializing the $6$-focal constraints 
to $\Rhoaqp^{m,n},$ we obtain
\begin{equation}\label{eq:det-dual-fundamental}
\det \begin{bmatrix}
E_1^\top \otimes I & e_1\\
E_2^\top \otimes I &    & e_2 \\
E_3^\top \otimes I &    &    & e_3 \\
E_4^\top \otimes I &    &   &&  e_4\\
\bar{Q}_{j_1} &  &  &   & & p_{i j_1}\\ 
\bar{Q}_{j_2} & & & &  & & p_{i j_2}
\end{bmatrix} 
=
0,
\quad 
\forall i = 1, \ldots , m, 
\, \, 
1\le j_1 < j_2 \le n.
\end{equation}
Permuting the rows,~\eqref{eq:det-dual-fundamental} implies that
\begin{align}
&\det 
\left[ 
\begin{array}{c|cc}
I_{12 \times 12} & \begin{array}{cccc}
E_1 & & & E_4\\
 & E_2 & & E_4\\
& & E_3 & E_4
\end{array} \\[.9em]
\hline 
\vspace{.9em}
\begin{array}{c}
\qbarcheck{j_1}\\
\qbarcheck{j_2}
\end{array}
& & 
\begin{array}{cc}
p_{i j_1} & \\
 & p_{i j_2}
\end{array}
\end{array}
\right]
= 0\nonumber ,
\end{align} 
and taking the Schur complement, we find
\begin{align}
&\det \left( 
\left[\begin{array}{ccc}
0_{3 \times 4} &p_{i j_1} & \\
0_{3\times 4} & & p_{i j_2}
\end{array}\right]
- 
\left[\begin{array}{c}
\qbarcheck{j_1}\\
\qbarcheck{j_2}
\end{array}\right]
\, 
\left[ 
\begin{array}{cccc}
E_1  & & & E_4\\
 & E_2 & & E_4\\
& & E_3 & E_4
\end{array}
\right]
\right) = \nonumber \\[0.8em]
&\det 
\left[\begin{array}{ccc}
A(\bar{q}_{j_1}) & p_{i j_1} & \\
A(\bar{q}_{j_2}) & & p_{i j_2}
\end{array}\right] = 0. \label{eq:dual-fundamental-1} 
\end{align}
\Cref{eq:dual-fundamental-1} is a bilinear form in $p_{i j_1}$ and $p_{i j_2},$ which may be represented by Carlsson and Weinshall's $3\times 3$ \emph{dual fundamental matrix} (cf.~\cite[eq.~18]{DBLP:journals/ijcv/CarlssonW98}),
\begin{equation}\label{eq:dual-fundamental-2}
\left[ \begin{smallmatrix}
0 & 
\bar{q}_{j_1} [2] \bar{q}_{j_2} [1] 
\det \left[\begin{smallmatrix}
\bar{q}_{j_1} [4] & \bar{q}_{j_1} [3] \\
\bar{q}_{j_2} [4] & \bar{q}_{j_2} [3]
\end{smallmatrix}\right] 
&
\bar{q}_{j_1} [3] \bar{q}_{j_2} [1] 
\det \left[\begin{smallmatrix}
\bar{q}_{j_1} [2] & \bar{q}_{j_1} [4] \\
\bar{q}_{j_2} [2] & \bar{q}_{j_2} [4]
\end{smallmatrix}\right] \\
\bar{q}_{j_1} [1] \bar{q}_{j_2} [2] 
\det \left[\begin{smallmatrix}
\bar{q}_{j_1} [3] & \bar{q}_{j_1} [4] \\
\bar{q}_{j_2} [3] & \bar{q}_{j_2} [4]
\end{smallmatrix}\right] 
& 
0 
& 
\bar{q}_{j_1} [3] \bar{q}_{j_2} [2] 
\det \left[\begin{smallmatrix}
\bar{q}_{j_1} [4] & \bar{q}_{j_1} [1] \\
\bar{q}_{j_2} [4] & \bar{q}_{j_2} [1]
\end{smallmatrix}\right] 
\\
\bar{q}_{j_1} [1] \bar{q}_{j_2} [3] 
\det \left[\begin{smallmatrix}
\bar{q}_{j_1} [4] & \bar{q}_{j_1} [2] \\
\bar{q}_{j_2} [4] & \bar{q}_{j_2} [2]
\end{smallmatrix}\right] 
& 
\bar{q}_{j_1} [2] \bar{q}_{j_2} [3] 
\det \left[\begin{smallmatrix}
\bar{q}_{j_1} [1] & \bar{q}_{j_1} [4] \\
\bar{q}_{j_2} [1] & \bar{q}_{j_2} [4]
\end{smallmatrix}\right] 
& 
0
\end{smallmatrix}
\right].
\end{equation}
This construction yields a total of $m \binom{n}{2}$ bilinear equations vanishing on the reduced joint image $\Rho_{\bar{\qq}, \pp}^{m,n}.$
A similar application of this Schur complement trick to suitably-chosen $7$- and $8$-focals leads to the \emph{dual trifocal and quadrifocal tensors} (cf.~\cite[\S 6.3--6.4]{DBLP:journals/ijcv/CarlssonW98},~\cite[\S 3, 4]{DBLP:conf/cvpr/TragerHP19}).

For a sufficiently generic point configuration $\bar{\qq} \in \left( \PP^3 \right)^n,$ it turns out that the reduced $2$-focals~\eqref{eq:dual-fundamental-1} determine $\Rho_{\bar{\qq}, \pp}^{m,n}$ as a subscheme of $\left(\PP^2\right)^{mn}.$
\Cref{thm:reduced-scheme-theoretic} states precise genericity conditions such that this occurs.
Thus, while equations needed to cut out $\GammaAqp^{m,n}$ in a strong sense have very high degree, only bilinear equations are needed to cut out its quotient $\Rhoaqp^{m,n}$ in a weaker sense.
The essential insight is, via Carlsson-Weinshall duality, that $\Rho_{\bar{\qq}, \pp}^{m,n}$ is simply the direct product of ``ordinary" multiview varieties,
\begin{equation}\label{eq:iso-joint-image}
\Rho_{\bar{\qq}, \pp}^{m,n} = 
\Gamma_{A(\bar{\qq}), \pp}^{n,m} \cong \Gamma_{A(\bar{\qq}), \pp}^{n,1} \times \cdots \times \Gamma_{A(\bar{\qq}), \pp}^{n,1}
\quad 
\text{ where } 
A(\bar{\qq}) = \left(A(\bar{q}_1), \ldots , A(\bar{q}_n) \right).
\end{equation}
A previous result, namely part (3) of~\Cref{thm:multiview-omnibus}, states that the multiview variety of a sufficiently generic camera arrangement is cut out by the bilinear forms in its vanishing ideal.
The Cremona transformation $\mathcal{C}$ allows us to translate these genericity conditions on the cameras $A(\bar{\qq})$ into conditions on the point arrangement $\bar{\qq} \in \left( \PP^2 \right).$
A $4$-nodal cubic surface in $\PP^3$ containing the points $E_1,\ldots E_4$ is given by an equation of the form
\begin{equation}\label{eq:cayley-cubic}
a_1 x_2 x_3 x_4 + a_2 x_1 x_3 x_4 + a_3 x_1 x_2 x_4 + a_4 x_1 x_2 x_3 = 0,
\quad 
[a_1 : a_2 : a_3 : a_4] \in \PP^3.
\end{equation}
If all $a_i$ are nonzero, then such a surface is projectively equivalent to \emph{Cayley's nodal cubic surface}, for which $a_1=a_2=a_3=a_4=1$ and the points $E_1,\ldots , E_4$ comprise the singular locus.
We also allow degenerate cases where one or more $a_i=0$ in~\eqref{eq:cayley-cubic}, in which case the surface degenerates to the union of a plane and a quadric, or the union of three planes.
\begin{theorem}\label{thm:reduced-scheme-theoretic}
    Fix $n$ distinct points $\bar q_1, \dots, \bar q_{n} \in \PP^3 \setminus \{ E_1, E_2, E_3, E_4\} $, $n \geq 2$, such that no four $\bar q_j$ lie on a common 4-nodal cubic surface through $E_1,\ldots E_4.$
    Write $$\bar{\qq} = (\bar{q}_1, \ldots , \bar{q}_{n}, E_1, \ldots , E_4),$$ and $\bar{\qq}'$ for the sub-arrangement of $\bar{\qq}$ obtained by deleting the $E_1, \ldots , E_4$.
    We have a birational equivalence of varieties
    \[
\Gammaqbarp^{m,n+4} \simeq_{\text{\bf Bir}} \Rho_{\bar{\qq} ', \pp}^{m,n} \times (\PGL_3)^m,
    \]
which realizes the reduced resectioning variety $\Rho_{\bar{\qq} ', \pp}^{m,n}$ as a rational quotient of $\Gammaqbarp^{m,n}$ by $(\PGL_3)^m$. 
Additionally, $\Rho_{\bar{\qq} ', \pp}^{m,n}$ is cut out scheme-theoretically by the $m \binom{n}{2}$ bilinear equations~\eqref{eq:dual-fundamental-1}.
\end{theorem}
\begin{proof}
The statements involving rational quotients follow similarly as in~\Cref{thm:cw-main-thm}.
Under the isomorphism~\eqref{eq:iso-joint-image}, the bilinear constraints in the theorem statement are the usual $2$-focals vanishing on the multiview variety.
We recall from part (3) of~\Cref{thm:multiview-omnibus} that the $2$-focals cut out the multiview variety scheme-theoretically whenever the camera centers are distinct and do not lie on a common plane.
Now, since $\bar{q}_i $ is not in the span of any three $E_j,$ the center of the camera $A(\bar{q}_i)$ is given by the Cremona transformation $\mathcal{C} (\bar{q}_i)$.
Since $\mathcal{C}$ maps any plane in $\PP^3$ to a $4$-nodal cubic surface, and vice-versa, we are done. 
\end{proof}
From the practitioner's point of view, the genericity assumptions of~\Cref{thm:reduced-scheme-theoretic}, as well as the implicit assumption that we can fix four fiducial 3D points and their images to the standard positions~\eqref{eq:standard-position}, may be quite reasonable.
This is supported by the experiments of~\cite[\S 5]{DBLP:conf/cvpr/TragerHP19}, suggesting some potential uses of Carlsson-Weinshall duality in SfM settings.

\section{Proof of~\Cref{thm:main-ideals}}\label{sec:ideals}
Our proof of~\Cref{thm:main-ideals} follows the general strategy used in the proof of~\cite[Theorem 3.2]{agarwal2022atlas}, but requires some nontrivial modifications.

\begin{remark}\label{remark:m-equals-one}
Unlike triangulation, resectioning is an interesting problem even for $m=1$ camera.
In fact, most of the work needed to prove~\Cref{thm:main-ideals} involves the special case $m=1$. 
As in the proof of~\Cref{prop:set-theoretic}, we fix $m=1$ and write $p_i$ in place of $p_{1 i}.$
We also write $\foc \left( \bullet \right)$ in place of $\Ifoc{1} \left( \bullet \right),$ and $\Gammaqbarp$ instead of $\Gammaqbarp^{1,n}.$
Finally, let us recall the variety $\Gamma_{\Qbar, \pp}^{n,m} \subset \left( \PP^2 \right)^{mn}$ introduced in~\Cref{prop:set-theoretic}.
In place of $\Gamma_{\Qbar, \pp}^{n,1},$ we simply write $\Gamma_{\Qbar, \pp}.$
\end{remark}

\subsection{Proof outline and preliminary facts}\label{subsec:outline}

To begin, we describe our proof strategy at a high level.
The main steps of our proof can be understood via the diagram in Figure \ref{fig:proof-schematic},
with each of the steps (1)--(4) explained below.
\begin{figure}[h]
\centering 
    \begin{tikzcd}[row sep = huge]
        & \textit{structured} 
        & \phantom{x         } 
        & \textit{generic} 
        \\[-2.5em]
        \textit{focal ideals} 
        & \foc (\bar\bQ) \arrow[rr, "(1)"', "\mathbf H \cdot", crossing over] \arrow[d, equal, "(4)"'] 
        & \phantom{hi} & \foc (\bar \bB)
        \arrow[d, equal, "(3)"]\\
        \textit{vanishing ideals} 
        & I(\Gammaqbarp) & \phantom{y} & I(\Gamma_{\bar\bB, \pp}) \arrow[ll, "(2)"', "\mathbf H^{-1} \cdot", crossing over] \vspace{.5em}
    \end{tikzcd}
    \caption{Schematic outline of the proof of~\Cref{thm:main-ideals}.}\label{fig:proof-schematic}
\end{figure}
\begin{enumerate}
    \item[(1)] \textbf{Coordinate change to obtain a generic hypercamera arrangement from the structured one.}
    For $\bar q_1, \dots, \bar q_n \in \PP^3$ with no four coplanar, 
    we establish in~\Cref{lemma:make-minor-generic,lem:rowspan-uniform-2} that there exist $3 \times 3$ invertible matrices $H_1, \dots, H_n$ such that the transformed hypercamera arrangement $$\bar{\bB} := (H_1 \bar Q_{1}  , \ldots, H_n \bar Q_{n})$$ is \emph{minor-generic} in the sense of~\Cref{def:minor-generic}.
    Applying the coordinate change $\mathbf H = (H_1, \dots, H_n)$ to $(\PP^2)^n$ reduces the study of $\foc (\qqbar)$ for the structured arrangement $\bar\bQ $ to that of $\foc (\bar{\bB})$ for the generic $\bar{\bB}.$
    \item[(2)] If we apply the inverse of the coordinate change $\mathbf H^{-1} = (H_1^{-1}, \dots, H_n^{-1})$ from step (1) to $(\PP^2)^n$, we can specialize from $I(\Gamma_{\bar \bB, \pp})$ to $I(\Gammaqbarp)$:
    \[
        \bar B_j (A) = p_j \quad \iff \quad 
        H_j^{-1} \bar B_j (A) = H_j^{-1} p_j \quad \iff \quad
        \bar Q_{j}(A) = H_j^{-1} p_j.
    \]
    \item[(3)] \textbf{Show equality of focal and vanishing ideals in the generic case.}
    By the previous two steps, it is sufficient to show 
    $I_{\operatorname{foc}}(\bar \bB) =
    I(\Gamma_{\bar\bB, \pp})$.
    We establish this using Gr\"obner bases, as described in~\Cref{subsec:gb-tools}.
    \item[(4)] \textbf{Show equality of focal and vanishing ideals in the structured case.} Combine steps (1)--(3).
\end{enumerate}
For the first step in the proof outline, we need the following definition.
\begin{definition}\label{def:minor-generic}
We say the hypercamera arrangement $\bar{\bB} = (\bar{B}_1, \ldots , \bar{B}_n) \in \left( \PP^{35} \right)^n$ is \emph{minor-generic} if all $12 \times 12$ minors of the $12 \times 3n$ matrix $\left( \bar{B}_1^\top \mid \cdots \mid \bar{B}_n^\top \right)$ are nonzero.
\end{definition}
This is a direct analogue of the genericity condition in~\Cref{thm:multiview-omnibus}, part (1).
We also need the following result.
Let $\FF$ be a field, and consider $s$ matrices $A_1, \ldots , A_s \in \FF^{M\times N}$. 
We say $A_1, \ldots , A_s$ are \emph{rowspan-uniform} if, for any subset $S\subset [s]$ of size at least $M/N,$  we have
\begin{equation}\label{eq:rowspan}
\displaystyle\sum_{i\in S} \rowspan (A_i) = \FF^N.
\end{equation}
\begin{lemma}\label{lemma:make-minor-generic}
If $A_1, \ldots , A_s \in \FF^{M\times N}$ are rowspan-uniform, then there exists a dense Zariski-open set of matrices $(H_1, \ldots , H_s) \in \GL (\FF^M)^s$ such that the maximal $n\times n$ minors of the the $sM \times N$ matrix
\begin{equation}\label{eq:stacked-transformed-matrix}
\left( 
\begin{array}{c}
H_1 A_1 \\ 
\hline 
\vdots \\
\hline 
H_s A_s
\end{array}
\right)
\end{equation}
are all nonzero.
\end{lemma}
We leave the proof of this result to~\Cref{appendix:proofs}.
This result is a direct generalization of~\cite[Lemma 3.6]{idealMultiview}, in the setting of triangulation.
In our setting of resectioning, we take $(M,N) = (3, 12),$ and deduce that we can transform the arrangement $\Qbar $ for suitably generic $\qqbar$ to a minor-generic arrangement $\bar{\bB}$ using the following result. 
\begin{lemma}\label{lem:rowspan-uniform-2}
Suppose that $\qqbar \in (\PP^3)^n$ is a point arrangement such that no four points are coplanar. Then $\Qbar $ is rowspan-uniform.
\end{lemma}
\begin{proof}
For any subset $S \subset [n]$ of size at least $4$, we must show \[\sum_{j \in S} \rowspan(\bar Q_j) = \CC^{12}  .\] 
Noting the compatible direct-sum decompositions
\begin{align*}
\CC^{12} &\simeq \CC^4 \oplus \CC^4 \oplus \CC^4 ,\\ 
\rowspan (\bar{Q}_j ) &\simeq \rowspan (\bar{q}_j^\top) \oplus \rowspan (\bar{q}_j^\top) \oplus \rowspan (\bar{q}_j^\top) ,
\end{align*}
it suffices to observe that any set of four elements from the set $\{ \bar{q}_j \}_{j \in S}$ span $\mathbb{P}^3,$ from our assumption that such a set is noncoplanar.
\end{proof}
\Cref{lem:rowspan-uniform-2} gives us a geometric interpretation of when $\Qbar $ is minor-generic. 
Algebraically, this condition is precisely what we need to obtain the Gr\"{o}bner basis of~\Cref{thm:main-ideals} via a standard specialization argument.
This is the focus of the next subsection.

\subsection{Gr\"{o}bner basis tools}\label{subsec:gb-tools}

To realize $\foc (\qqbar )$ as the specialization of an ideal that is independent of $\qqbar$, we could replace the arrangement $\Qbar$ with $\bB = (B_1, \ldots , B_n)$, where
\begin{equation}\label{eq:Q-matrix}
B_i = 
\left[
\begin{smallmatrix}
B_i [1,1] & 
B_i [1,2] &
B_i [1,3] & 
B_i [1,4] & 
B_i [1,5] & 
B_i [1,6] & 
B_i [1,7] & 
B_i [1,8] & 
B_i [1,9] & 
B_i [1,10] & 
B_i [1,11] & 
B_i [1,12] \\
B_i [2,1] & 
B_i [2,2] &
B_i [2,3] & 
B_i [2,4] & 
B_i [2,5] & 
B_i [2,6] & 
B_i [2,7] & 
B_i [2,8] & 
B_i [2,9] & 
B_i [2,10] & 
B_i [2,11] & 
B_i [2,12] \\
B_i [3,1] & 
B_i [3,2] &
B_i [3,3] & 
B_i [3,4] & 
B_i [3,5] & 
B_i [3,6] & 
B_i [3,7] & 
B_i [3,8] & 
B_i [3,9] & 
B_i [3,10] & 
B_i [3,11] & 
B_i [3,12] 
\end{smallmatrix}\right],
\end{equation}
thereby introducing $36 n$ new indeterminates.
Alternatively, we could replace $\Qbar$ with the symbolic arrangement $\bB^\star = (B_1^\star, \ldots , B_n^\star)$, where
\begin{equation}\label{eq:Qvee-matrix}
B_i^\star  = 
\left[
\begin{smallmatrix}
B_i^\star [1,1] & 
B_i^\star [1,2] &
B_i^\star [1,3] & 
B_i^\star [1,4] & 0 & 0 &0 &0 & 0 & 0 & 0 &0 \\
0 & 0 &0 &0 & B_i^\star [2,1] & 
B_i^\star [2,2] &
B_i^\star [2,3] & 
B_i^\star [2,4] & 0 & 0 & 0 & 0\\
0 & 0 &0 &0 & 0 & 0 & 0 &0 & B_i [3,1]^\star & 
B_i^\star [3,2] &
B_i^\star [3,3] & 
B_i^\star [3,4]
\end{smallmatrix}\right],
\end{equation}
for a total of $12n$ new indeterminates.
For either $\bB$ or $\bB^\star $ and for each $k$ with $6\le k \le m$, we are interested in the determinants of the matrices
\begin{align}
(\bB \, | \, \pp ) [\mathbf{r} ]_\si &= 
\begin{bmatrix} B_{\si_1}[\row{1}, :] &p_{\si_1}[\row{1}] &0&\dots &0\\B_{\si_2}[\row{2}, :] & 0 & p_{\si_2}[\row{2}]&\ddots &0 \\ \vdots & \vdots &\ddots&\ddots &\vdots \\B_{\si_k}[\row{k}, :] &0&\dots&0&p_{\si_k}[\row{k}]
\end{bmatrix}, \label{eq:Q-focal-matrix}
\\[0.5em]
(\bB^\star \, |  \, \pp ) [\mathbf{r} ]_\si &=
\begin{bmatrix} B_{\si_1}^\star[\row{1}, :] &p_{\si_1}[\row{1}] &0&\dots &0\\B_{\si_2}^\star[\row{2}, :] & 0 & p_{\si_2}[\row{2}]&\ddots &0 \\ \vdots & \vdots &\ddots&\ddots &\vdots \\B_{\si_k}^\star[\row{k}, :] &0&\dots&0&p_{\si_k}[\row{k}]
\end{bmatrix},\label{eq:Qvee-focal-matrix} \\
\text{where } 
\mathbf{r} = (\row{1}, \ldots , \row{n}) &\subset [3]^n, \nonumber \\ 
\# \row{1} + \cdots + \# \row{k} &= k , \nonumber \\
\si = \{ \si_1, \ldots , \si_k \} &\subset [n]. \nonumber 
\end{align}
Upon specializing $\bB \to \Qbar$, or $\bB^\star \to \Qbar$, the respective determinants of~\eqref{eq:Q-focal-matrix} or~\eqref{eq:Qvee-focal-matrix} specialize to the same $k$-focal.

\begin{proposition}\label{prop:bump}
Equip either ring $\CC [ \bB , \pp ]$ or $\CC [\bB^\star , \pp ]$ with the $\ZZ^{2n}$-grading defined on generators by $\deg (B_i [j,k] ) = \deg (B_i^\star [j,k] ) = e_i $ and $\deg (p_{i} [j]) = e_{n+i}.$
The polynomial $\det (\bB \, | \, \pp ) [\mathbf{r} ]_\si$ is homogeneous of multidegree 
\begin{equation}\label{eq:multidegree}
\sum_{i=1}^k (\# \mathbf{r}_i - 1)e_{\sigma_i} + e_{n + \sigma_i}.
\end{equation}
The same is true for $\det (\bB^\star \, | \, \pp ) [\mathbf{r} ]_\si,$ provided it is nonzero.
When $k>12,$ we have $\mathbf{r}_j = \{ l \}$ for some $j\in [k]$, $l \in [3],$ and hence
\begin{align}
\det (\bB \, | \, \pp ) [\mathbf{r} ]_\si &= p_{\sigma_j} [l] \cdot \det (\bB \, | \, \pp ) [\mathbf{r}_1, \ldots , \widehat{\mathbf{r}_j}, \ldots , \mathbf{r}_k ]_{\si \setminus \{ j \}} , \label{eq:focal-factor}\\
\det (\bB^\star \, | \, \pp ) [\mathbf{r} ]_\si &= p_{\sigma_j} [l] \cdot \det (\bB^\star \, | \, \pp ) [\mathbf{r}_1, \ldots , \widehat{\mathbf{r}_j}, \ldots , \mathbf{r}_k ]_{\si \setminus \{ j \}} .\nonumber 
\end{align}
\end{proposition}
\begin{proof}
The argument is nearly identical to~\cite[Proposition 3.3]{agarwal2022atlas}.
The multidegree formula~\eqref{eq:multidegree} follows from a calculation analagous to that already given in~\eqref{eq:rescale}.
Since $
12 = \displaystyle\sum_{i=1}^k (\# \row{i} - 1),$
it follows that at most $12$ of the sets $\row{i}$ can contain more than one element. 
If some $\row{j}$ is a singleton, the factorization~\eqref{eq:focal-factor} follows by Laplace expansion.
\end{proof}
We now define four auxiliary ideals.
\begin{definition}\label{def:Q-focal-ideals}
The ideals $\sixB, \mB \subset \CC [\bB , \pp ]$ are those which are generated by all determinants of~\eqref{eq:Q-focal-matrix} for all $k$, respectively, in the ranges $6\le k \le 12,$ and $k=m.$
Similarly, $\sixBstar, \mBstar \subset \CC [\bB^\star , \pp ]$ are generated by all determinants of~\eqref{eq:Qvee-focal-matrix}.
\end{definition}

Each of the four auxiliary ideals in~\Cref{def:Q-focal-ideals} is useful for different reasons.
For example, $\mB $ and $\mBstar$ are the ideals of maximal minors of a \emph{sparse generic matrix} whose nonzero entries are distinct indeterminates.
Thus, $\mB$ and $\mBstar$ belong to the class of \emph{sparse determinantal ideals}, whose structure has been analyzed in several previous works~\cite{giusti, boocher}. 
Most relevant to our work is the result of~\cite{boocher} which directly implies that the $m$-focals form a universal Gr\"{o}bner basis for either of these ideals.

For the other two ideals $\sixB$, $\sixBstar$, we do not know whether or not the focals form universal Gr\"{o}bner bases.
However,~\Cref{prop:gb-six-focals} shows that they \emph{do} form Gr\"{o}bner bases for a class of product orders that allows us to make the necessary specialization argument.

We recall that a \emph{product order} on $\CC [\mathbf{B} , \mathbf{p} ]$ with $\mathbf{B} < \mathbf{p}$ is a monomial order defined by comparing monomials first with some fixed monomial order in $\mathbf{p} $, then breaking any ties with some other monomial order in $\mathbf{B}.$

\begin{proposition}\label{prop:gb-six-focals}
\phantom{f}
\begin{enumerate}
\item The set $G = \{ \det (\bB \, |  \, \pp ) [\mathbf{r}]_\sigma \mid 6\le \# \sigma \le 12\}$ forms a Gr\"{o}bner basis for the ideal $\sixB$ for any product order with $\bB < \pp .$
\item The set $G^\star = \{ \det (\bB^\star \, |  \, \pp ) [\mathbf{r}]_\sigma \mid 6\le \# \sigma \le 12\}$ forms a Gr\"{o}bner basis for the ideal $\sixBstar$ for any product order with $\bB^\star < \pp .$
\end{enumerate}
\end{proposition}
We provide a proof in~\Cref{appendix:proofs}.
A specialization argument applied to the two parts of~\Cref{prop:gb-six-focals} gives, respectively, the two parts of~\Cref{lem:gb-specialized} below.
\begin{lemma}\label{lem:gb-specialized}
For any monomial order on $\CC [\pp ],$ the following hold:
\begin{enumerate}
\item Let $\bf{\bar{B}}$ be a specialization of $\bf{B}$ such that the $\bf{\bar{B}}$ is minor generic. Then the specialized $6-12$ focals form a Gr\"{o}bner basis for the ideal they generate.
\item Let $\Qbar$ be a specialization of $\bf{B}^\star$ derived from a point arrangement $\qqbar \in \left( \PP^3 \right)$ with no four points coplanar. Then the specialized $6-12$ focals form a Gr\"{o}bner basis.
\end{enumerate}

\end{lemma}
\begin{proof}For part (1), let $<$ be any product order with $\mathbf{B}<\mathbf{p}$. Then
\begin{equation}
    \text{in}_< \left(\sum_{p^{\alpha_1}<\ldots<p^{\alpha_k}}g_{\alpha_i}(\mathbf{B})\mathbf{p}^{\alpha_i}\right)=\text{in}_<(g_{\alpha_k}(\mathbf{B}))p^{\alpha_k}.
\end{equation}
Standard specialization results for Gr\"{o}bner bases with respect to product orders \cite[Theorem 2, \S 4.7]{CLO} imply that the $\bf{\bar{B}}$ specialized $6-12$ focals form a Gr\"{o}bner basis if each coefficient $g_{\alpha_k}(\bf{\bar{B}})$ is nonzero. Each of these coefficients is a $12\times 12$ minor of $(\bar{B}_1^\top \mid \cdots \mid \bar{B}_n^\top)$. By minor-genericity, none of these coefficients vanish.

Similarly, for part (2), consider any product order $<$ with $\bB^*<\mathbf{p}$. 
In this case, the nonzero coefficients $g_{\alpha_k}(\bB^\star) $ are always products of three $4\times 4$ determinants,
\begin{equation}\label{eq:det-prod}
\displaystyle\prod_{i=1}^3 \det 
\begin{bmatrix}
B_{j_{i, 1}}^\ast [i,1] & B_{j_{i, 1}}^\ast [i,2] & B_{j_{i, 1}}^\ast [i, 3] & B_{j_{i, 1}}^\ast [i,4] \\
B_{j_{i, 2}}^\ast [i,1] & B_{j_{i, 2}}^\ast [i,2] & B_{j_{i, 2}}^\ast [i, 3] & B_{j_{i, 2}}^\ast [i,4] \\
B_{j_{i, 3}}^\ast [i,1] & B_{j_{i, 3}}^\ast [i,2] & B_{j_{i, 3}}^\ast [i, 3] & B_{j_{i, 3}}^\ast [i,4] \\
B_{j_{i, 4}}^\ast [i,1] & B_{j_{i, 4}}^\ast [i,2] & B_{j_{i, 4}}^\ast [i, 3] & B_{j_{i, 4}}^\ast [i,4] 
\end{bmatrix}.
\end{equation}
Our noncoplanarity assumption implies that the $\bB^\star \to \Qbar $  specialization of~\eqref{eq:det-prod} is nonzero.
\end{proof}
Finally, we have the following result on the vanishing ideal of $\Gamma_{\bar{\bB}, \pp}$ when $\bar{\bB}$ is a minor-generic hypercamera arrangement.
See~\Cref{appendix:proofs} for the proof.
\begin{proposition}\label{prop:vanishing-ideal-minor-generic}
For a minor-generic hypercamera arrangement $\bar{\bB}$, we have that 
\[
I(\Gamma_{\bar{\bB}, \pp}) = \foc (\bar{\bB}).
\]
\end{proposition}

\subsection{Completing the proof}\label{subsec:proof-main-ideals}

Using the results of the previous sections, we may complete the proof of~\Cref{thm:main-ideals}, following the overall structure presented in~\Cref{fig:proof-schematic}.

We first prove the statement for $m=1$ camera.
Let $\bar{\qq} \in \left(\PP^3 \right)^m$ be a point arrangement with no four points coplanar.
\Cref{lem:rowspan-uniform-2} then implies that the hypercamera arrangement $\bar{\bQ}$ is rowspan-uniform, and thus~\Cref{lemma:make-minor-generic} implies there exist coordinate changes in the images, $\mathbf{H} = (H_1, \ldots , H_m ) \in (\PGL_3)^m$, such that the arrangement $\bar{\bB} = (H_1 \bar Q_{1}, \ldots , H_m \bar Q_{m})$ is minor-generic.
Noting
\begin{equation}\label{eq:focal-transform}
\begin{pmatrix}
\bar Q_{1} & p_1 & &\\
\vdots & & \ddots &\\
\bar Q_{n} & & & p_n
\end{pmatrix}
=
\begin{pmatrix}
H_1^{-1} & &\\
&\ddots &\\
& & H_n^{-1}
\end{pmatrix}
\begin{pmatrix}
\bar{B}_1 & H_1 p_1 & &\\
\vdots & & \ddots &\\
\bar{B}_n & & & H_n p_n
\end{pmatrix},
\end{equation}
we define the isomorphism of multigraded rings \begin{align*}
L_{\mathbf{H}} : \CC [\pp] &\to \CC [\pp] \\
p_i &\to H_i p_i,
\end{align*}
and observe that
\[
\foc (\qqbar) = L_{\mathbf{H}} (\foc (\bar{\bB})) .
\]
To see this, take any focal $f \in \foc (\qqbar )$.
Just as in the proof of~\Cref{lemma:make-minor-generic}, corresponding minor of the focal matrix on the left of~\Cref{eq:focal-transform} may written as a $\mathbb{C}$-linear combination of focals for the arrangement $\bar{\bB}.$ 
Hence the inclusion $\foc (\qqbar)\subset L_{\mathbf{H}} (\foc (\bar{\bB})) $ holds, and the reverse follows similarly.
Thus, we have
\begin{align*}
\foc (\qqbar) &= L_{\mathbf{H}} (\foc (\bar{\bB})) \\
&= L_{\mathbf{H}} ( I(\Gamma_{\bar{\bB}, \pp})) \tag{\Cref{prop:vanishing-ideal-minor-generic}}\\
&= I(\Gamma_{\Qbar , \pp}) \\
&= I(\Gamma_{\qqbar , \pp}) \tag{\Cref{prop:set-theoretic}}. 
\end{align*}
Thus the focals generate $I(\Gamma_{\qqbar , \pp})$.
Moreover,~\Cref{lem:gb-specialized} part (2) implies that they form a universal Gr\"{o}bner basis, which completes the proof when $m=1$.

Finally, if $m>1,$ it suffices to observe that $\Gamma_{\qqbar, \pp}^{m,n}$ is the direct product of varieties $\Gamma_{\qqbar, \pp}$, and hence the vanishing ideals sum.
Moreover, two $k$-focals corresponding to different factors have disjoint support in $\CC [\pp],$ so their S-polynomials reduce to zero for any term order, and we may conclude that the focals form a universal Gr\"{o}bner basis for any number of cameras.

\section{Optimal single-camera resectioning}\label{sec:resectioning}

The results of~\Cref{sec:CWduality} express a duality principle for the exact versions of the camera resectioning and triangulation problems.
A consequence of this duality is that, in a certain sense, resectioning and triangulation are equivalent problems.
However, we should be mindful that this equivalence holds in an idealized setting which assumes that the pinhole camera is exact and there is no measurement noise.
In practice, neither of these assumptions hold.

In this section, we fix a generic point arrangement $\bar{\qq}$ and consider $\Gamma_{\qqbar, \mathbf{p}}^{1,n} \subset (\PP^2)^n$ intersected with the affine chart where $p_i[3]\ne 0$ for all $1\le i \le n.$ 
We denote this affine variety by $X_{\qqbar, n}$.
In other words, for a point arrangement $\qqbar \in \left( \PP^3\right)^n$ such that no four points are coplanar, the affine variety $X_{\qqbar, n}$ is, by~\Cref{prop:set-theoretic}, equal to the closed image of the rational map
\begin{align}
\psi_{\qqbar , n} : \PP^{11} &\dashrightarrow \CC^{2n} \nonumber \\
A &\mapsto \left(
\displaystyle\frac{A[1, :] \bar{q}_1}{ A[3, :] \bar{q}_1},
\displaystyle\frac{A[2, :] \bar{q}_1}{ A[3, :] \bar{q}_1},
\ldots ,
\displaystyle\frac{A[1, :] \bar{q}_n}{ A[3, :] \bar{q}_n},
\displaystyle\frac{A[2, :] \bar{q}_n}{ A[3, :] \bar{q}_n}
\right). \label{eq:psi}
\end{align}
In the resectioning problem, we are given world points $\bar{\qq} = (\bar{q}_1, \ldots , \bar{q}_n)\in (\PP^3)^n$ and pixel values of $n$ corresponding image points, $(\tilde{u}_i, \tilde{v}_i) $ for $i=1, \ldots , n.$
We denote the vector of image measurement data by $\tilde{d}_{uv} = (\tilde{u}_1, \ldots , \tilde{v}_n) \in \CC^{2n}$.
In practice, $\tilde{d}_{uv}$ and $\qqbar$ are both defined over the real numbers.
Our task is to recover a camera $A$ such that 
\begin{equation}\label{eq:exact-resection}
\psi_{\qqbar, n} (A) = \tilde{d}_{uv}.
\end{equation}
In an idealized setting, the pinhole model is exact and there is no measurement noise.
Hence, we can recover $A$ by computing the kernel of the $n$-focal matrix, and we expect a unique solution as soon as $n\ge 6.$
This is the basis of the so-called ``5.5-point" minimal solver.

In practice, the pinhole model is \emph{not} exact and there \emph{is} measurement noise.
Thus, for $n\ge 6,$ we should expect $\tilde{d}_{uv} \notin X_{\qqbar, n}$, meaning that no solution to~\eqref{eq:exact-resection} can exist.
However, we can still consider the following optimization problem: 
\begin{equation}\label{eq:opt-resection}
L_{\tilde{d}_{uv}}(u_1, v_1, \ldots , u_n, v_n) = \displaystyle\sum_{i=1}^n
(u_i - \tilde{u}_i)^2 + (v_i - \tilde{v}_i)^2
\quad 
\text{s.t.}
\quad 
(u_1, \ldots , v_n) \in X_{\qqbar, n}.
\end{equation}
This is essentially the formulation of the optimal resectioning problem that is used in Hartley and Zisserman's classic text~\cite{HZ04}[\S 7.2].
The only minor difference, implicit in their formulation, is that our feasible set $X_{\qqbar, n}$ differs from theirs by a set of measure zero picked up through Zariski closures.
Similar formulations, which make the camera matrix explicit, appear in other sources, eg.~in work of Cifuentes~\cite[Example 6.5]{cifuentes2021convex} who studied sums-of-squares relaxations of this problem.
Hartley and Zisserman refer to the squared Euclidean loss function $L_{\tilde{d}_{uv}}$ as the \emph{geometric error}, and suggest using local methods like Levenberg-Marquardt to optimize it.
Here, we address the complexity of computing the \emph{global minimum} of~\eqref{eq:opt-resection}.

We recall the notion of the \emph{Euclidean distance degree} of an affine variety,~\cite[\S 2]{draisma2016euclidean}.
For $X_{\qqbar, n},$ we denote this quantity by $\ED (X_{\qqbar, n})$. 
Given a generic data point $\tilde{d}_{uv}$, this is the number of critical points of the squared Euclidean loss $L_{\tilde{d}_{uv}}$ restricted to the smooth locus of $X_{\qqbar, n}$.

\begin{conjecture}\label{conjecture:ED}
For all $n\ge 6$ and generic $\qqbar \in (\PP^3)^n,$ we have 
\begin{equation}\label{eq:ED-formula-conjecture}
\ED (X_{\qqbar, n}) = (80/3) n^3 - 368 n^2 + (5068/3) n - 2580.
\end{equation}
\end{conjecture}
We return to~\Cref{ex:hypersurface-case-1}, to verify the simplest case of this conjecture. 
\begin{example}\label{ex:hypersurface-case-2}
Consider the resectioning hypersurface $H(u_1, \ldots , v_6) =0$, ie.~\eqref{eq:det-hypersurface} in the chart
\begin{equation}\label{eq:ex-chart}
p_1[3] = p_2 [3] = p_3[3] = p_4[3] = p_5[3] = p_6[3] = 1.\end{equation}
The affine variety $X_{\bar{\qq}, 6} \subset \CC^{12}$ is in fact the cone over a projective variety in $\PP^{11}.$ 
This can be seen from the determinantal representation of $H$ in~\eqref{eq:12-det}.
It follows that $X_{\bar{\qq}, n}$ is singular.
More precisely,  the singular locus of $X_{\bar{\qq}, n}$ has dimension $9.$

Working over the finite field $\mathbb{F} = \mathbb{Z}_{32003}$, we may verify~\Cref{conjecture:ED} with symbolic computation using the computer algebra system Macaulay2~\cite{M2}.
To do so, we draw a $\mathbb{F}$-valued point configuration $\qqbar \in \left( \PP^3 \right)^6$ and data vector $\tilde{d}_{uv} \in \FF^{12}$ uniformly at random.
The critical points of~\eqref{eq:opt-resection} correspond to points $(u_1, \ldots,  v_6) \in X_{\qqbar, 6}$ such that 
\begin{equation}\label{eq:rank-deficient-ED}
\rank  \begin{bmatrix}
u_1 - \tilde{u}_1 & \cdots & v_6 - \tilde{v}_6 \\
\frac{\partial H}{\partial u_1} & \cdots & \frac{\partial H}{\partial v_6}
\end{bmatrix}
\le 1.   
\end{equation}
To remove the singular points on $X_{\bar{\qq}, 6}$ which cause rank-deficiency in~\eqref{eq:rank-deficient-ED}, it is sufficient take the ideal generated by the $2\times 2$ minors of this matrix and $H(u_1, \ldots , v_6)$ and compute its ideal quotient with respect to the ideal $\langle \frac{\partial H}{\partial u_1}, \frac{\partial H}{\partial v_1} \rangle $.
The result of this operation is a zero-dimensional ideal of degree $68.$
Moreover, we may compute that the vanishing locus of this ideal consists of $68$ distinct, nonsingular points on $X_{\bar{\qq}, 6}.$
The number $68$ may be seen as quantifying the intrinsic algebraic difficulty of solving the constrained optimization problem~\eqref{eq:opt-resection}.
This is further reinforced by heuristically computing the Galois/monodromy group of this problem, as in~\cite{galvis1}, which reveals the full symmetric group $S_{68}.$
\end{example}
Our conjectural formula~\eqref{eq:ED-formula-conjecture} is reminiscent of recent results characterizing the Euclidean distance degree of the {\em affine multiview variety} $X_{\AAbar, m}$. 
This can be defined by taking analagous affine charts on the multiview variety $\GammaAbarp^{m,1}$.
Using a topological formula for the ED-degree of a smooth variety, Maxim, Rodriguez, and Wang~\cite{MRW20} proved
\begin{equation}\label{eq:ED-formula-multiview}
\ED (X_{\AAbar, m}) = (9/2) m^3 - (21/2) m^2 + 8 m - 4.
\end{equation}
We compare this formula with ours in~\Cref{tab:ED-comparison}.
We confirmed the entries of this table using numerical monodromy heuristics~\cite{monodromy1}, using both the implementations provided in Macaulay2~\cite{M2} and Julia~\cite{HCJL}.
For these computations, it is advantageous to use the rational parametrization~\eqref{eq:psi} instead of the implicit focal constraints in~\Cref{thm:main-ideals}.

A surprising aspect of~\Cref{conjecture:ED} is that $\ED (X_{\qqbar , n})$ is a polynomial of degree 3 in $n.$
On the other hand, if we were to apply the methods of~\cite{MRW20} to computing the affine ED-degree of the variety $\Gamma_{\bar{\bB }, \qq}^{n,1}$ associated to a \emph{generic} hypercamera arrangement $\bar{\bB} \in \left( \PP^{11} \right)$, this would give instead a polynomial of degree $11.$
This highlights some special properties of the hypercamera arrangement $\Qbar $, and provides contrast with the results of previous sections.
One explanation for this contrast is the fact that the projective coordinate changes used in~\Cref{thm:cw-main-thm,thm:reduced-scheme-theoretic} do not preserve the Euclidean distance.
For similar reasons, the affine ED degree of the reduced resectioning variety, which is the same as $\ED (X_{\AAbar, m})$, appears to be unrelated to that of the general resectioning variety.

\begin{table}
\begin{center}
\begin{tabular}{c|c|c}
$m$ / $n$ & $\ED (X_{\AAbar , m})$ &  $\ED (X_{\qqbar, n})$ \\
\hline 
2 & 6 & --- \\
3 & 47 & --- \\
4 & 148 & --- \\
5 & 336 & --- \\
6 & 638 & 68 \\
7 & 1081 & 360 \\
8 & 1692 & 1036 \\
9 & 2498 & 2256 \\
10 & 3526 & 4180 \\
11 & 4803 & 6968 \\
12 & 6356 & 10780 \\
13 & 8212 & 15776 \\
14 & 10398 & 22116 \\
15 & 12941 & 29960
\end{tabular}
\end{center}
\caption{Euclidean distance degrees for optimal triangulation from $m$ generic cameras (middle column) and optimal resectioning from $n$ 3D points (right.)}\label{tab:ED-comparison}
\end{table}

We close this section by noting one immediate obstacle to proving~\Cref{conjecture:ED}.
As already seen in~\Cref{ex:hypersurface-case-2}, the variety $X_{\qqbar , n}$ for generic data $\bar{\qq}$ is \emph{not smooth} for any $n \ge 6.$
This contrasts with the case of $X_{\AAbar , m}$, which is smooth for a sufficiently generic arrangement of $m\ge 3$ cameras $\bar{\bA}.$
Thus, to prove~\eqref{eq:ED-formula-multiview} with similar techniques, the basic Euler characteristic formulas valid in the smooth case would need to be replaced by their singular counterparts involving Euler obstruction functions, eg.~\cite[Theorem 1.3]{eulerobs}.

\section{Conclusion}\label{sec:conclusion}

In summary, our work takes several first steps in studying the resectioning problem for general projective cameras from the algebro-geometric perspective, with a focus on
Gr\"{o}bner bases, Carlsson-Weinshall duality, and Euclidean distance optimization.
Our discoveries provide many parallels with the already well-studied multiview ideals associated with the triangulation problem.
Still, many open questions remain.

In this paper, we considered resectioning in the setting of general projective cameras.
Returning to the classical P3P problem~\cite{Grunert-1841}, it would be worthwhile to carry out a parallel study in the setting of \emph{Euclidean cameras}, as proposed in~\cite[\S 8.3, Q1]{agarwal2022atlas}.
In view of~\Cref{thm:multiview-omnibus} parts (2)--(3), it is natural to ask: are all $k$-focals for $6\le k \le 12$ are needed to generate $I_m (\qqbar )$ under the noncoplanarity assumption of~\Cref{thm:main-ideals}?
What can we say about $I_m (\qqbar )$ if this noncoplanarity assumption is relaxed?
Using the reduced atlas developed~\Cref{sec:CWduality} to answer more of the open questions in~\cite[\S 8]{agarwal2022atlas} is yet another interesting avenue to pursue.
Our focus on resectioning for linear maps $\mathbb{P}^3 \dashrightarrow \mathbb{P}^2$ was motivated by computer vision. 
However, it would make just as much sense to study resectioning varieties in the context of general projections $\mathbb{P}^N \dashrightarrow \mathbb{P}^M,$~\cite{Li-IMRN}, or even matrix multiplication maps as in~\cite[\S 8.3, Q3]{agarwal2022atlas}.
Finally, we offer~\Cref{conjecture:ED} as a challenge in Euclidean distance degree computation.

\section*{Acknowledgements}
The authors thank Sameer Agarwal, Max Lieblich, and Rekha Thomas for many helpful conversations and suggestions.
TD also thanks Laurentiu Maxim, Jose Rodriguez, and Felix Rydell for helpful discussions related to~\Cref{sec:resectioning}, and acknowledges support from an NSF Mathematical Sciences Postdoctoral Research Fellowship (DMS-2103310).

\bibliographystyle{amsplain}
\bibliography{refs.bib}

\appendix
\section{Miscellaneous Proofs}\label{appendix:proofs}
First, we prove~\Cref{lemma:make-minor-generic}, justifying the coordinate change $\mathbf{H}$ used to prove~\Cref{thm:main-ideals}.
\begin{proof}[Proof of~\Cref{lemma:make-minor-generic}]\label{proof-make-minor-generic}
Consider some maximal minor of the matrix~\eqref{eq:stacked-transformed-matrix}.
We fix the set of indices $\{ \sigma_1, \ldots , \sigma_k \},$ where $1\le \sigma_1 < \sigma_2 < \cdots < \sigma_k \le s $, such that at least one row is taken from the submatrix $H_{\sigma_j} A_{\sigma_j}$ when forming this minor, and let $1 \le i_{j 1}, \ldots , i_{j r_j} \le N$ index the rows that are taken from this submatrix.
We compute this minor using the multilinearity of the determinant:
\begin{align*}
\det \left( \begin{array}{c}
H_{\sigma_1} A_{\sigma_1} [i_{1 1}, :] \\
\hline 
\vdots \\
\hline 
H_{\sigma_k} A_{\sigma_k} [i_{k r_k}, :] \end{array}
\right) &= 
\det \left( \begin{array}{c}
\displaystyle\sum_{l=1}^M H_{\sigma_1} [i_{1 1}, \ell ] A_{\sigma_1} [ \ell , : ] \\ 
\hline 
\vdots \\
\hline 
\displaystyle\sum_{l=1}^M H_{\sigma_k} [i_{k r_k}, \ell ] A_{\sigma_k} [ \ell , : ]
\end{array}
\right) \\
&= \displaystyle\sum_{1 \le \ell_1, \ldots , \ell_M \le M}
\det \left( \begin{array}{c}
A_{\sigma_1} [\ell_{1}, :] \\ 
\hline 
\vdots \\
\hline 
A_{\sigma_k} [\ell_{M}, :] \end{array}
\right)
\cdot 
\left(H_{\sigma_1} [i_{1 1}, \ell_{1}] \cdots H_{\sigma_k} [i_{k r_k}, \ell_{M}] \right).
\end{align*}
We think of this minor as a polynomial in the entries of $(H_1, \ldots , H_s).$
Our assumption of rowspan-uniformity 
implies that $A_{\sigma_1} [\ell_1, :], \ldots , A_{\sigma_k} [\ell_M, :]$ form a basis of $\FF^N$ for some choice of indices in the sum above, and hence one of the coefficients of this polynomial is nonzero.
Thus, the equation
\begin{equation}\label{eq:transformed-minor-zero}
\det \left( \begin{array}{c}
H_{\sigma_1} A_{\sigma_1} [i_{1 1}, :] \\ 
\hline 
\vdots \\
\hline 
H_{\sigma_k} A_{\sigma_k} [i_{k r_k}, :] \end{array}
\right) = 0
\end{equation}
defines a hypersurface in the affine space of all $k$-tuples of $M\times M$ matrices $(H_1, \ldots , H_s).$
Taking the union over all such hypersurfaces and those defined by $\det H_i=0$ gives us a proper Zariski-closed set $Z$ in this affine space.
The complement of $Z$ is an open set which satisfies the desired conclusion.
\end{proof}
Next, to prove~\Cref{prop:gb-six-focals}, we recall~\cite[Definition 3.6]{agarwal2022atlas}.
\begin{definition}
Consider a polynomial $f\in \CC [\bB , \pp_{\sigma_1}, \ldots , \pp_{\sigma_k}]$ which is homogeneous of degree $1$ in each group of variables $\pp_{\sigma_i} = \{ p_{\sigma_i} [1], p_{\sigma_i} [2], p_{\sigma_i} [3] \} .$ We say $f$ is \emph{well-supported}  with respect to $\pp_{\sigma_1}, \ldots , \pp_{\sigma_k}$ if, for every choice of variables $p_{\sigma_1} [i_1]\in \pp_{\sigma_1} , \ldots , p_{\si_k} [i_k]\in \pp_{\si_k} $ such that each $p_i [i_j]$ appears in some term of $f$, the monomial $\displaystyle\prod_{j=1}^k p_{\sigma_j} [i_j]$ also appears in $f$ (with nonzero coefficient in $\CC [\bB]$.)
\end{definition}
More intuitively, $f$ is well-supported if its monomial support in $\pp$ is as large as possible given its variable support, or if it has a dense coefficient tensor in $\CC [ \bB ].$
For a product order with $\bB < \pp $, well-supportedness implies that the leading terms of $k$-focals depend only on the relative orderings of variables within each of the groups $\pp_1, \ldots , \pp_n $ and the ordering on $\bB$ (cf.~\cite[Lemma 3.8]{agarwal2022atlas}.)
When $6\le k\le m,$ an argument using Laplace expansion can be used to show that the $k$-focal $\det (\bB \, | \, \pp ) [\mathbf{r} ]_\si$ is well-supported with respect to $\pp_{\sigma_1}, \ldots , \pp_{\sigma_k}.$

In contrast, the $k$-focal $\det (\bB^\star \, | \, \pp ) [\mathbf{r} ]_\si$ is \emph{not} well-supported with respect to $\pp_{\sigma_1}, \ldots , \pp_{\sigma_k},$ since the nonzero coefficients of $\pp$-monomials must have the form~\eqref{eq:det-prod}.

\begin{proof}[Proof of~\Cref{prop:gb-six-focals}]
For part (1), we construct an ascending chain of ideals 
\[
J_0 \subset J_1 \subset \cdots \subset J_n,
\]
where $J_k = \langle G_k \rangle $, and $G_k$ is an inductively-defined Gr\"{o}bner basis with respect to the appropriate class of product orders.
We take $G_0$ to be the set of $m$-focals, so that $J_0$ is a sparse determinantal ideal. 
As previously noted,~\cite[Proposition 5.4]{boocher} implies that $G_0$ is a \emph{universal} Gr\"{o}bner basis.
Having defined $G_k$ for some $k\ge 0,$ we define $G_{k+1}$ to be the set consisting of all polynomials $g$ such that either $g\in G_k$ and is not divisible by any entry of $p_k$ or such that $g\notin G_k$ with $p_k[l] \cdot g \in G_k$ for some $l\in [3].$
\Cref{prop:bump} implies that each $G_k$ may be obtained from $G_0$ by dividing out any entry of the matrices $p_1, \ldots , p_k$ from any $m$-focal containing it as a factor.
When $k=n,$ we obtain $G=G_n$ as the set of all $k$-focals for $6\le k \le 12.$
We claim that each $G_k$ is a Gr\"{o}bner basis for the appropriate class of product orders, and moreover that $J_k$ can be expressed in terms of ideal quotients as
\begin{equation}\label{eq:Jk-sat}
J_k = J_{k-1} : \langle p_k [1] \rangle = J_{k-1} : \langle p_k [2] \rangle = J_{k-1} : \langle p_k [3] \rangle .
\end{equation}
The elements of each $G_k$ are well-supported, and hence by~\cite[Corollary 3.9]{agarwal2022atlas} the Gr\'{o}bner basis property for product orders is preserved.

Having established part (1), we can prove part (2) via an argument used in the proof of ~\cite[Proposition 5.4]{boocher}. If $<$ is any product order with $\mathbf{B}^\star<\mathbf{p}$ then we can extend this to a product order with $\mathbf{B}<\mathbf{p}$ where the entries of $\mathbf{B}$ which are zero in $\mathbf{B^\star}$ are weighted last.

Let $f\in\sixBstar$ be nonzero, so that $f=\sum c_{\sigma , \mathbf{r}} \det (\bB^\star \, |  \, \pp) [\mathbf{r}]_\sigma$ for some coefficients $c_{\sigma , \mathbf{R}} \in \CC [\bB^\star , \pp ] \subset \CC [\bB , \pp].$ 
Consider the lifted polynomial 
\[
\bar{f}=\sum c_{\sigma , \mathbf{r}}(\bB , \pp) \,   \det (\bB \, |  \, \pp ) [\mathbf{r}]_\sigma \in \sixB .
\]
 Our chosen weighting implies that $\text{in}(f)=\text{in}(\bar f)$. 
Part (1) implies $\text{in}(\bar f)$ is divisible by the leading monomial $\bar m_{\sigma , \mathbf{r}}=\text{in}(\det (\bB \, |  \, \pp ) [\mathbf{r}]_\sigma)$ corresponding to some summand above. 
It follows that $\text{in}(f)$ is divisible by $m_{\sigma , \mathbf{r}} =\text{in}(\det (\bB^\star \, |  \, \pp ) [\mathbf{r}]_\sigma)$. This gives part (2).
\end{proof}

\begin{proof}[Proof of~\Cref{prop:vanishing-ideal-minor-generic}]
Having shown the set-theoretic statement in~\Cref{prop:set-theoretic}, it is enough to show that the focal ideal is radical and saturated with respect to the irrelevant ideal of $\left( \PP^2 \right)^n.$ 
Radicality follows from~\Cref{lem:gb-specialized}, since the initial ideal is squarefree. 
For saturatedness, we note that, in the notation of the previous proof, the focal ideal $I (\bar{\bB})$ is the specialization of the ideal $J_m.$
Using~\eqref{eq:Jk-sat} with $k=m$ and the fact that specialization $\bB \to \bar{\bB}$ preserves the Gr\"{o}bner basis property, it follows that $I (\bar{\bB}) : \langle p_k [i] \rangle $ for all $i$ and $j.$
This in turn implies saturatedness with respect to the irrelevant ideal.
\end{proof}
\Addresses
\end{document}